\documentclass[11pt]{amsart}   
\usepackage{amsmath}
\usepackage{amsthm}
\usepackage{mathdots}
\usepackage[pagebackref,colorlinks=true,linkcolor=blue,urlcolor=blue]{hyperref}
\usepackage{stmaryrd}
\usepackage[initials, lite]{amsrefs}
\usepackage{color}
\usepackage[onehalfspacing]{setspace}
\usepackage{tabularx}
\usepackage{amsfonts} 
\usepackage{paralist}
\usepackage{aliascnt}
\usepackage{amscd}
\usepackage{blkarray}
\usepackage{mathbbol}
\usepackage{setspace}
\usepackage[inner=2.5cm,outer=2.5cm, bottom=3.2cm]{geometry}
\usepackage{tikz, tikz-cd}
\usepackage{amsmath,amscd,amsthm,amssymb,amsxtra,latexsym, epsfig,epic,graphics}
\usepackage{tikz}
\usetikzlibrary{matrix}
\usetikzlibrary{arrows,calc}
\allowdisplaybreaks

\BibSpec{collection.article}{%
	+{}  {\PrintAuthors}                {author}
	+{,} { \textit}                     {title}
	+{.} { }                            {part}
	+{:} { \textit}                     {subtitle}
	+{,} { \PrintContributions}         {contribution}
	+{,} { \PrintConference}            {conference}
	+{}  {\PrintBook}                   {book}
	+{,} { }                            {booktitle}
	+{,} { }                            {series}
	+{, vol.} { }                            {volume}
	+{,} { }                            {publisher}
	+{,} { \PrintDateB}                 {date}
	+{,} { pp.~}                        {pages}
	+{,} { }                            {status}
	+{,} { \PrintDOI}                   {doi}
	+{,} { available at \eprint}        {eprint}
	+{}  { \parenthesize}               {language}
	+{}  { \PrintTranslation}           {translation}
	+{;} { \PrintReprint}               {reprint}
	+{.} { }                            {note}
	+{.} {}                             {transition}
	+{}  {\SentenceSpace \PrintReviews} {review}
}
\AtBeginDocument{%
	\def\MR#1{}
}

\newcommand{\vsup}{\mathrel{\rotatebox{90}{$\supset$}}}

\newcommand{\veq}{\mathrel{\rotatebox{90}{$=$}}}
\newcommand{\vcong}{\mathrel{\rotatebox{90}{$\cong$}}}

\newcommand{\m}{\mathfrak{m}}

\newcommand{\rank}{\normalfont\text{rank}}
\newcommand{\reg}{\normalfont\text{reg}}
\newcommand{\sat}{{\normalfont\text{sat}}}
\newcommand{\depth}{\normalfont\text{depth}}
\newcommand{\grade}{\normalfont\text{grade}}

\newcommand{\Ext}{\normalfont\text{Ext}}

\newcommand{\Supp}{\normalfont\text{Supp}}

\newcommand{\Hom}{{\normalfont\text{Hom}}}

\newcommand{\Fitt}{\normalfont\text{Fitt}}

\newcommand{\B}{\mathcal{B}}

\newcommand{\indeg}{\normalfont\text{indeg}}

\newcommand{\Spec}{{\normalfont\text{Spec}}}

\newtheorem{theorem}{Theorem}[section]

\newaliascnt{headcor}{headthm}

\aliascntresetthe{headcor}

\newaliascnt{headconj}{headthm}

\aliascntresetthe{headconj}

\newaliascnt{corollary}{theorem}
\newtheorem{corollary}[corollary]{Corollary}
\aliascntresetthe{corollary}

\def\ms{\medskip}
\def\Der{\operatorname{Der}}

\def\Lmult{\operatorname{Lmult}}
\def\Sing{\operatorname{Sing}}

\def\coker{\operatorname{coker}}
\def\Hom{\operatorname{Hom}}

\def\Quot{\operatorname{Quot}}
\def\Spec{\operatorname{Spec}}
\def\ann{\operatorname{ann}}

\def\Der{\operatorname{Der}}

\def\Ext{\operatorname{Ext}}
\def\socle{\operatorname{socle}}

\def\im{\operatorname{im}}
\def\fitt{\operatorname{Fitt}}

\def\min{\operatorname{min}}
\def\inf{\operatorname{inf}}
\def\reg{\operatorname{reg}}

\def\edim{\operatorname{edim}}
\def\dim{\operatorname{dim}}

\def\depth{\operatorname{depth}}
\def\grade{\operatorname{grade}}
\def\rank{\operatorname{rank}}
\def\htt{\operatorname{ht}}

\def\findeg{\operatorname{findeg}}
\def\deg{\operatorname{deg}}
\def\indeg{\operatorname{indeg}}
\def\trdg{\operatorname{trdeg}}

\def\Z{\mathcal Z}
\def\H{\mathcal H}

\def\D{\mathcal D}

\def\J{\mathcal J}

\def\C{\mathcal C}

\def\q{\mathfrak q}
\def\p{\mathfrak p}

\def\gC{\mathfrak C}
\def\gc{\mathfrak c}
\def\b{\mathfrak b}
\def\a{\mathfrak a}
\def\m{\mathfrak m}
\def\f{\mathfrak f}
\def\p{\mathfrak p}
\def\n{\mathfrak n}
\def\N{\mathfrak N}

\def\lto{\longrightarrow}

\def\bs{\bigskip}
\def\ms{\medskip}

\newaliascnt{lemma}{theorem}
\newtheorem{lemma}[lemma]{Lemma}
\aliascntresetthe{lemma}

\newaliascnt{conjecture}{theorem}

\aliascntresetthe{conjecture}

\newaliascnt{proposition}{theorem}
\newtheorem{proposition}[proposition]{Proposition}
\aliascntresetthe{proposition}

\theoremstyle{definition}
\newaliascnt{definition}{theorem}
\newtheorem{definition}[definition]{Definition}
\aliascntresetthe{definition}

\newaliascnt{notation}{theorem}

\aliascntresetthe{notation}

\newaliascnt{example}{theorem}
\newtheorem{example}[example]{Example}
\aliascntresetthe{example}

\newaliascnt{examples}{theorem}

\aliascntresetthe{examples}

\newaliascnt{remark}{theorem}
\newtheorem{remark}[remark]{Remark}
\aliascntresetthe{remark}

\newaliascnt{question}{theorem}

\aliascntresetthe{question}

\newaliascnt{questions}{theorem}

\aliascntresetthe{questions}

\newaliascnt{problem}{theorem}

\aliascntresetthe{problem}

\newaliascnt{construction}{theorem}

\aliascntresetthe{construction}

\newaliascnt{setting}{theorem}
\newtheorem{setting}[setting]{Setting}
\aliascntresetthe{setting}

\newaliascnt{algorithm}{theorem}

\aliascntresetthe{algorithm}

\newaliascnt{observation}{theorem}

\aliascntresetthe{observation}

\newaliascnt{defprop}{theorem}

\aliascntresetthe{defprop}

\DeclareFontFamily{OT1}{pzc}{}
\DeclareFontShape{OT1}{pzc}{m}{it}{<-> s * [1.100] pzcmi7t}{}
\DeclareMathAlphabet{\mathchanc}{OT1}{pzc}{m}{it}

\def\equationautorefname~#1\null{(#1)\null}
\def\sectionautorefname~#1\null{Section #1\null}
\def\subsectionautorefname~#1\null{\S #1\null}

\def\trdeg{{\rm trdeg}}

\begin{document}

\baselineskip=16pt

\title[Bounds on degrees of Vector Fields]{ 
Bounds on degrees of Vector Fields}
\date\today

\thanks{AMS 2010 {\em Mathematics Subject Classification}.
Primary 14C20, 13A30; Secondary 14M10, 14J17.}


\thanks{The second author was partially supported by CNPq grant (No. 406377/2021-9) and CAPES grant PrInt. The third and fifth authors were partially supported by NSF grants DMS-2201110 and DMS-2201149, respectively.}
\author{Marc Chardin}\address{Institut de Math\'ematiques de Jussieu, CNRS \& Sorbonne Universit\'e, France}\email{mchardin@imj-prg.fr}

\author{S. Hamid Hassanzadeh}\address{Departamento de Matem\'atica, Centro de Tecnologia, Cidade Universit\'{a}ria da Universidade Federal do Rio de Janeiro, 21941-909 Rio de Janeiro, RJ, Brazil}\email{hamid@im.ufrj.br}

\author{Claudia Polini}
\address{Department of Mathematics, University of Notre Dame, Notre Dame, Indiana 46556} \email{cpolini@nd.edu}

\author{Aron Simis}\address{Departamento de Matem\'atica, CCEN, Universidade Federal de Pernambuco, Recife, PE, 50740-560, Brazil}\email{simisaron0@gmail.com}

\author{Bernd Ulrich}
\address{Department of Mathematics, Purdue University, West Lafayette, Indiana 47907} \email{bulrich@purdue.edu}
 \maketitle
\section{Introduction}

\vspace{.15cm}

This paper is concerned with the structure of the module of derivations and its interplay with singularities and vector fields of varieties. 
Modules of derivations are not well understood -- despite great advances on the Zariski-Lipman conjecture (see \cites{Ho, SvS, Flen} for instance),  there is still no complete characterization for when they are free.
The paper focuses on Poincar\'{e}'s problem on the degrees of vector fields.

In 1891, Poincar\'{e} asked the following question that became known as Poincar{\'e}'s problem \cite{P}:  
How can one decide whether a homogeneous differential equation given by a polynomial vector field
 $\mathcal F$
on 
$\mathbb P^2_{\mathbb C}$ 
has a rational solution?  This question has been rephrased as the problem to find {\it upper} bounds for the degree of any curve $\mathcal C$ to which $\mathcal F$ is tangent, possibly in terms of the degree of $\mathcal F$.

According to \cite{EK2003}*{p.57}, `This question is fundamental but difficult, and it has stimulated a lot of research for well over a century' (see \cites{CL, C, CC, S1, dPW, CCF, S2, E, LN, Per, EK, EK2002, EK2003, CE, GM}). It has often been addressed 
in greater generality, for curves and even varieties in ${\mathbb P}_{\mathbb C}^n$,
and invariants other than the degrees of the variety and the vector field 
have been considered, because 
even for plane curves
bounds only involving degrees are not always possible \cites{C,CC,LN}.

In this article, we study the generalized Poincar\'e problem from 
the opposite perspective, by establishing
{\it lower} bounds on the degree of the vector field in terms of invariants of the variety, say $X$. This approach 
has the advantage that all the vector fields on ${\mathbb P}^n_k$ tangent to $X$ can be encoded in a 
single module, 

\centerline{$\Der_k(R)/\m^{-1}\varepsilon\, .$}

\vspace{0.18cm}
\noindent
Here $R$ is the homogeneous coordinate ring of a 
subscheme $X \subset \mathbb P^n_k $, which, for the purpose of this introduction, is assumed to be reduced and irreducible over an 
algebraically closed field of characteristic zero; by $\m$ we denote the homogeneous maximal ideal of $R$, by $\Der_k(R)$ the module of derivations, and by 
$\varepsilon \in \Der_k(R)$ the Euler derivation. If $\depth \, R \geq 2$, then $\m^{-1}\varepsilon =R\varepsilon$
and the module above essentially carries the same information as $\Der_k(R)$. The least degree of a vector field that leaves $X$ invariant
and does not vanish along $X$ is 1 plus the initial degree of the module $\, \Der_k(R)/\m^{-1}\varepsilon$, and our reformulation of Poincar\'{e}'s problem 
becomes:  {\it  Find lower bounds for the initial degree $\, {\rm indeg}(\Der_k(R)/\m^{-1}\varepsilon)$. 
}

In the current paper we address this problem mainly for curves.
We
generalize bounds that were known for plane curves and we obtain new estimates as well. 
Our proofs are algebraic. In order to understand how tight the lower bounds for the initial 
degree of $\Der_k(R)/\m^{-1}\varepsilon$ are,
we also provide upper bounds, which sometimes lead to equalities. Our estimates
use global invariants, such as the genus of the curve $\mathcal C$, the Castenuovo-Mumford regularity,
or the $a$-invariant of the homogeneous
coordinate ring $R$; invariants
that can be considered global as well as local, like the singularity degree of $\mathcal C$, the Tjurina number, or 
the multiplicity of $R$ modulo the Jacobian ideal; 
and local
information, such as the type of the singularities or a new invariant that we call {\it Loewy multiplicity}.

\vspace{.15cm}

For smooth curves we prove that the initial
degree satisfies the inequality ${\rm indeg}(\Der_k(R)/\m^{-1}\varepsilon)
\geq a(R)+1,$ which is an equality if $\mathcal C$ is
arithmetically Gorenstein. In one of our main results, \autoref{curves}, we generalize this
inequality to the case of curves with at most planar singularities. We show that
\begin{equation}\label{introEQ}
\indeg (\Der_k(R)/\m^{-1}\varepsilon )\geq a(R) +1 + | {\rm Sing}(\C) | - {\rm Lmult}(R/ J_R) \, .
\end{equation}
Here $a(R)=-\, \indeg (\omega_R) \, $ denotes the  $a$-invariant of $R$, which is 
equal to ${\rm reg}  \ \mathcal C -3 \, $ if $\mathcal C$ is arithmetically Cohen-Macaulay; 
and ${\rm Lmult}(R/J_R)$ denotes the Loewy multiplicity of $R$ modulo the Jacobian ideal, which is bounded above by
the sum of the local Tjurina numbers of $\mathcal C$ in this case.
If $\mathcal C$ has only ordinary nodes as singularities, then 
$|{\rm Sing}(\C) | - {\rm Lmult}(R/ {\rm Jac}(R))=0$ and we obtain
the inequality $\indeg  (\Der_k(R)/\m^{-1}\varepsilon ) \geq a(R) +1$, which is again an equality whenever $\mathcal C$ is arithmetically Gorenstein.
The case of ordinary nodes had
been treated before with the additional assumption that, first, $\C$ is a plane curve \cite{CL},
then, $\C$ is a complete intersection \cite{CCF}, and, finally, $\C$
is arithmetically Cohen-Macaulay \cites{E, EK2002}. 

The proof of inequality \autoref{introEQ} has two main ingredients.
Inspired by the use of general projections in \cite{E},
we prove more generally in \autoref{GP} that if $A \subset R$ is a finite and birational extension of standard graded domains  
over a perfect field and $A$ is Gorenstein of dimension at least two, then 
$$ \indeg (\Der_k(A)/A  \varepsilon_A) +a(A)-a(R) \geq  \indeg (\Der_k(R)/\m^{-1} \varepsilon_R) \geq \indeg (\Der_k(A)/A  \varepsilon_A) -a(A)+a(R) \, .$$
In addition, for hypersurfaces of arbitrary dimension with only isolated singularities, we are able to bound $\indeg (\Der_k(R)/\m^{-1}\varepsilon)$
from below in terms of the $a$-invariant of $R/J_R$. Applying these two results to 
a curve $\C$ with only planar singularities, we prove inequality \autoref{introEQ} by general projection to a plane curve. The general projection
does not change the singularities of $\C$ and introduces only ordinary nodes as additional singularities, which guarantees that the
difference $| {\rm Sing}(\C) | - {\rm Lmult}(R/J_R)$ is unaltered.

Our other main results generalize, from plane curves to arbitrary curves, earlier work of
du Plessis and Wall \cite{dPW} and of Esteves and Kleiman \cite{EK} that uses the sum of 
local Tjurina numbers. For a curve $\mathcal C \subset \mathbb P^n_k$ of degree
$d$ with homogeneous coordinate ring $R$, we map
a sufficiently general two-dimensional complete intersection $S$ onto $R$, and write $\delta$ for the Castelnuovo-Mumford
regularity of $S$ and $J\subset R$ for the image of the Jacobian ideal of $S$. 
If $\mathcal F$ is a vector field that leaves
$\mathcal C$ invariant and  
$0 \neq I_{\mathcal F} \subset R$ is an ideal defining the singular locus of $\mathcal F$, 
then $\mathcal F$ and $J$ coincide up to a degree shift and this shift is closely related to the degree 
${\rm deg} \, \mathcal F$, which we wish to control. The degree shift is reflected in the difference of multiplicities $e(R/I_{\mathcal F})-e(R/J)$. Thus 
to estimate the degree shift, and hence $\rm{deg} \, \mathcal F$, from below it suffices to establish a lower bound for
$e(R/I_{\mathcal F})$. To do so, we
prove a non-vanishing result for maps between local cohomology 
modules of Koszul cycles 
that yields lower bounds for the regularity
of the saturation $I_{\mathcal F}^{\rm sat}$ and hence for $e(R/I_{\mathcal F})$. The line of argument
just described is inspired by the work of Esteves and Kleiman in \cite{EK}. Thus we obtain in \autoref{dPW}
that
$$\indeg (\Der_k(R)/\m^{-1})\ge \delta -2 -\frac{e(R/J)-\delta}{d-1} \, ,$$
unless $\mathcal C$ is a smooth complete intersection, in which case the initial degree is $a(R)+1$.

This result is the starting point for various estimates in terms of the arithmetic genus $p_a$, the geometric genus $p_g$, and
the sum $\tau$ of the local Tjurina numbers of the curve. \autoref{turina} says that  
$$\indeg (\Der_k(R)/\m^{-1}\varepsilon)\ge \frac{d}{d-1} \, a(R) -\frac{\tau-2}{d-1}$$
if $\mathcal C$ is locally a complete intersection, and
$$\indeg (\Der_k(R)/R\varepsilon)\ge \frac{2 p_a -\tau}{d-1}$$
if in addition $\mathcal C$ is arithmetically Cohen-Macaulay. For plane curves we prove in \autoref{dPW-EK} that
$$\indeg (\Der_k(R)/R\varepsilon )\geq d-\frac{3}{2} - \sqrt{2\tau +2p_g-d^2+3d-\frac{7}{4}} \,.
$$
These bounds are sharp for plane curves of low genus and for other classes of curves, as illustrated in  \autoref{beauty}.

In \autoref{SecUB} we turn to upper bounds for $\indeg (\Der_k(R)/\m^{-1}\varepsilon)$ in order to
understand how sharp the lower bounds in \autoref{SecEK} and \autoref{SecLL} are.
As a special case of \autoref{UB}, for instance, we prove that $a(R)+1$ is an upper bound whenever $\mathcal C$ is arithmetically
Gorenstein. In \autoref{curvein3} we determine the minimal graded free resolution of $\Der_k (R)/R\varepsilon$ as a module over a polynomial ring if 
$\mathcal C \subset \mathbb P_k^3$ is smooth and arithmetically Cohen-Macaulay. From this we obtain the initial degree, the minimal number of generators, and the entire Hilbert series of  $\Der_k (R)/R\varepsilon $. In particular, we see that the upper bound $a(R)+1$ fails dramatically without the assumption of arithmetic Gorensteinness. Curiously, 
the work in  \autoref{SecEK} on 
local cohomology of
Koszul cycles implies another result about the structure of the module of derivations -- namely we prove in \autoref{splitting} that the Euler
derivation cannot generate a free direct summand of $\Der_k(R)$ when $\mathcal C$ is arithmetically Gorenstein.

\vspace{.2cm}

\section{Preliminary results and a translation from geometry to algebra}\label{Section-translation}

\vspace{.1cm}

Let $R$ be a standard graded algebra over a field with homogeneous maximal ideal $\m$. Let $\Omega$ be the module of differentials of $R$ and let $\varepsilon$ denote the {\it Euler derivation}. In this section we prove that the vector fields studied in \cites{CC, CCF, C,E, EK, EK2002, EK2003, S1, S2} correspond to elements in the $R$-module $\Der_k(R)/\m^{-1}\varepsilon$, see \autoref{1.2}. 

We begin by reviewing basic definitions related to vector fields; we basically follow the definitions from the  excellent reference \cite{E}*{p. 4-5}. 

We adopt the following setting:

\begin{setting}\label{sett-vectfield}  Let $n \geq 2$, and $S=k[x_1, \ldots, x_n]$  the homogeneous coordinate ring of $\mathbb P^{n-1}_k$, with maximal homogeneous ideal $\m_S$. Let $I\subset S$ be a saturated homogeneous ideal and $R=S/I$ be the homogeneous coordinate ring of the corresponding projective scheme $X\subset {\mathbb P^{n-1}_k}$. Let $\m$ be the maximal homogeneous ideal of $R$ and $\varepsilon$ the Euler derivation. 
\end{setting}

The Euler sequence 
$$\begin{tikzcd} 0 \arrow{r} &Z \arrow{r} & \Omega_k(S)=\oplus_{i=1}^{n}Sdx_i \cong S^n(-1) \arrow{rr}{[x_1 \ \ldots \ x_n]} &&\m_S  \arrow{r} &0
 \end{tikzcd} $$
 defines the cotangent sheaf $\Omega_{\mathbb P^{n-1}_k} $ as $\widetilde{Z}$. Notice that $Z$ is the first syzygy module in the Koszul complex of $x_1, \ldots, x_n$.

A {\it vector field } on $\mathbb P_k^{n-1}$ of degree $m$ is a homogeneous  map of degree $m-1$
\[\eta: Z\lto S\, .\]
As $\Ext^1_S(\m_S,S)=0$, any such map is the restriction of a map
$$\begin{tikzcd}\xi: \Omega_k(S)\cong S^n(-1)  \arrow{rr}{[a_1 \ \ldots \ a_n]} &&S\, ,
 \end{tikzcd} $$
 where the $a_i$ are forms of degree $m$.


There is a commutative diagram with exact rows 
\begin{equation}\label{ZANDR}\begin{tikzcd}0 \arrow{r} &Z \arrow{r} \arrow{d}{\varphi} & \Omega_k(S)\cong S^n(-1)  \arrow[d, two heads, "\psi"]\arrow{r} &\m_S \arrow[d, two heads, "\varrho"] \arrow{r} &0 \\
0 \arrow{r} &L \arrow{r}  & \Omega_k(R) \arrow{r} &\m \arrow{r} &0 
 \end{tikzcd} \, .
 \end{equation}
We write  $H=\im \varphi$ and notice that this is the image of the second differential in the 
Koszul complex built on the $R$-linear map $\Omega_k(R) \lto R$ corresponding to the Euler derivation (by the universal property). In particular, $L/H$ is the first homology of this Koszul complex and hence it
is annihilated by $\m$. Moreover, $H=L$
  if $I$ is generated by forms whose degrees are not multiples of the characteristic. Indeed, in this case the Euler relations shows that  $\ker \, \psi$ maps onto  $\ker \, \varrho$, hence $\varphi$ is surjective by the Snake Lemma. 

One says that the vector field $\eta$ {\it leaves $X$ invariant} or that $X$ is an {\it integral subscheme} of $\eta$ (if $X$ is a curve we say that  $X$ is a {\it leaf} of $\eta$) if $\eta$ induces a  map $\mu: H \lto R$, necessarily  linear and homogeneous of degree $m-1$, 
$$\begin{tikzcd}
Z \arrow[hookrightarrow]{rr}  \arrow[two heads]{d} \arrow{ddr}{\eta} && S^n(-1)  \arrow{ddl}{\xi}    \\
H \arrow{ddr}{\mu} &&  \\
&S \arrow[two heads]{d} &\\
& \ R \, . & 
\end{tikzcd}  $$

Notice that a map $H \lto R$ corresponds to a unique map $L\lto R$ if $I$ is generated by forms whose degrees are not multiples of the characteristic or if $\depth R \ge 2$.

Summarizing, every vector fields $\eta$ of degree $m-1$ that leaves $X$ invariant induces a unique homogeneous $R$-linear map $\mu: H \lto R$ of degree $m-1$. 
\medskip

\begin{proposition}\label{translation} 
Adopt \autoref{sett-vectfield}. A homogeneous $R$-linear map $\mu:H \lto R$ is induced by a vector field that leaves $X$ invariant if and only if $\mu$ can be extended to a homogeneous 
 $R$-linear map $\nu: \Omega_k(R)\lto R$. 
 \end{proposition}
\begin{proof} 
Given a vector field $\eta: Z \lto S$, the map $\mu: H \lto R$ is induced by  $\eta$ if and only if 
$\mu$ is induced by $\eta \otimes_S R.$
Consider the commutative diagram with exact rows and columns
$$\begin{tikzcd}  
& V\arrow[d, hookrightarrow] \arrow{rr}{\tau} && U\arrow[d, hookrightarrow] \arrow{r} &I/\m_SI \arrow[d, hookrightarrow] &\\
&Z \otimes_SR\arrow{rr}  \arrow[d, two heads]  \arrow{dddr}{\eta \otimes R}& & R^n(-1)  \arrow[two heads]{dd} \arrow{r}  \arrow[dddl, "\xi \otimes R"']&\m_s\otimes_SR \arrow[two heads]{dd} \arrow{r} &0 \hphantom{\ .}\\
&H \arrow[ddr, "\mu" ' near end]\arrow[d, phantom, "\vsup"] &&&&\\
0 \arrow{r} &L \arrow{rr} & & \Omega_k(R) \arrow{r} \arrow{dl}{\nu}&\m \arrow{r} &  0  \ . \\
&&   R &&&
 \end{tikzcd} $$
If the map  $\mu$ is induced by  $\eta$, hence by $\eta\otimes R$,  then $ (\eta \otimes R)(V)=0$ which implies  $(\xi \otimes R)(\im \tau)=0$. By the above diagram, $\coker  \tau \hookrightarrow I/\m_SI$ and therefore $\m \cdot U\subset \im \tau$. Thus $\m \cdot (\xi \otimes R)(U)=0$, which implies that $(\xi \otimes R)(U)=0$ since $\depth \, R >0$. It follows that $\xi \otimes R$ induces a homogeneous $R$-linear map $\Omega_k(R) \lto R$, which gives $\mu$ when restricted to $H$. 
 
 Conversely, let $\nu: \Omega_k(R) \lto R$ be a homogeneous $R$-linear map. It can be lifted to a homogeneous $S$-linear map $\xi: S^n(-1)\lto S$ because $S^n(-1)$ is free. Set $\eta=\xi_{|Z}$. Since $(\xi\otimes R)(U)=0$, we have $(\eta\otimes R)(V)=0$ and so $\eta$ induces a homogeneous $R$-linear map $\mu: H\lto R$, which is also the restriction of $\nu$ to $H$. 
\end{proof}
 
We write $-^*=\Hom_R(-, R)$.  In light of \autoref{translation} we are interested in the image of $\Omega_k(R)^*=\Der_k(R)$ in the module $H^*$. In the next proposition, we identify this image. We use the fact that there is a natural embedding $ \Der_k(R) \hookrightarrow Q \otimes_R \Der_k(R) $ where $Q$ is the total ring of fractions of $R$. For an ideal $\a$ of $R$, we denote  its inverse ideal by  $\a^{-1}:=R :_Q \a$. Notice that if $\depth R \ge 2$ then $\m^{-1}=R$. 

 \begin{example} Let $\C$ be the rational quartic curve given by the parametrization 
\begin{tikzcd}[row sep={0.22em}, column sep={4em}]
 \mathbb{P}_k^1 \arrow[r, dashrightarrow, "(s^4:s^3t:st^3:t^4)"] & \mathbb{P}_k^3 \, .
\end{tikzcd}   
In this case $\m^{-1}$ is $\overline{R}$, the integral closure of $R,$
and $\overline{R}=R[s^2, st, t^2]\subset  Q.$
 \end{example}

\begin{proposition}\label{1.2}  
Adopt \autoref{sett-vectfield}. There are homogeneous exact sequences
\[0\lto \Der_k(R)/\m^{-1}\varepsilon \lto L^* \lto \Ext^2_R(k,R)\, ,
\]
\[0\lto L^* \lto H^* \lto \Ext^1_R(C, R)\, ,
\]
where $C$ is an $R$-module annihilated by $\m$. 

Therefore, there are natural homogeneous embeddings
\[\Der_k(R)/\m^{-1}\varepsilon \hookrightarrow L^{*} \hookrightarrow H^{*} , \] 
where the first embedding is an isomorphism if $\depth R \ge 3$ and the second is an isomorphism if $\depth R\ge 2$ or  the defining ideal of $R$ is generated by forms whose degrees are not multiples
of the characteristic. 

In particular, the vector fields that leave $X$ invariant correspond to the elements in the torsionfree $R$-module $\Der_k(R)/\m^{-1}$. 
\end{proposition}
\begin{proof} Consider the Euler sequence 
$$\begin{tikzcd} 
& & &  R &\\
0 \arrow{r} & L \arrow{r} & \Omega_k(R) \arrow{ru} \arrow{r} & \m  \arrow{u} \arrow{r} & 0\, .
 \end{tikzcd} $$
Dualizing into $R$ shows that  the first row of 
$$\begin{tikzcd} 0\arrow{r} & \m^{-1} \arrow{r} & \Der_k(R) \arrow{r} & L^*\arrow{r} & \Ext_R^1(\m,R) \\
&R\arrow{u} \arrow{ur}  &&
 \end{tikzcd} $$
 is exact and that $1\in R $ maps to $\varepsilon \in  \Der_k(R)$. 
Since $\Ext^1_R(\m,R)\cong \Ext^2_R(k,R)$, we obtain the first asserted exact sequence.

The second exact sequence follows because $C:=L/H =\coker \varphi$ is annihilated by $\m$ and $\depth \, R>0$, 
see page \pageref{ZANDR}. We also recall that $H=L$ if the defining ideal of $R$ is generated by forms whose degrees are not multiples
of the characteristic. 
\end{proof}

The {\it singular locus} of the vector field $\eta$ is the subscheme $ \Sigma=V(I_2(N))\subset \mathbb P^{n-1}_k$, where $N$ is the $2$ by $n$ matrix 
$$N=\begin{bmatrix}
x_1 & \ldots & x_n \\
a_1 & \ldots & a_n
\end{bmatrix}\, ;$$
in fact a point $P\in \mathbb P^{n-1}_k$ does not belong to $\Sigma$ if and only if $Q:=[a_1(P): \ldots: a_n(P)]\in \mathbb P^{n-1}_k$ and there is a unique line passing  through $P$ and $Q$, giving the direction defined by $\eta$ at $P$. We observe that $I_2(N)R$, the ideal defining the subscheme $\Sigma \cap X \subset X$, is the image of the map $\mu: H \lto R$ induced by $\eta$.
One usually requires that $\Sigma \cap X$ does not contain an irreducible component of $X$, in other words, that the ideal $\im \mu=I_2(N)R$ has positive height in $R$. 

\smallskip

We introduce a new invariant that is going to play an important role throughout the paper. 
\begin{definition} Let $R$ be a non-negatively graded ring and $M$ be a finitely generated $R$-module. We define the {\it faithful initial degree} of $M$ over $R$ as
\[\findeg_R M=\inf \{\deg m \ | \ m \in M \ \mbox{homogenous with }  \ann m=0\}\, .\]
\end{definition}

\vspace{.1cm}

Notice that $\findeg M\ge \indeg M$, and equality holds if $M$ is torsionfree and $R$ is a domain.

\vspace{.1cm}

\begin{corollary}   In addition to  \autoref{sett-vectfield}, 
assume that $R$ has no embedded associated primes.
If $m$ is the smallest degree of a vector field on $\mathbb P^{n-1}_k$ that leaves $X$ invariant and whose singular locus
does not contain an irreducible component of $X$, then $$m=1+\findeg( \Der_k(R)/\m^{-1} \varepsilon) \, . $$
\end{corollary}
\begin{proof} \autoref{translation} and \autoref{1.2} show that $m-1$ is the smallest degree of a homogenous element in $ \Der_k(R)/\m^{-1} \varepsilon$ that, when regarded as a homogenous $R$-linear map $H \lto R$, has the property that $\htt (\im \mu ) >0$, equivalently $\grade \im \mu >0$, or yet equivalently  $\ann_R \mu =0$. 
\end{proof}

In the next proposition (and the remark following it) we identify $H^*$ with a well-known fractional ideal: the inverse of the image in $R$ of the Jacobian ideal of a general complete intersection mapping onto $R$. The resulting embedding $\Der_k(R)/\m^{-1}\varepsilon\, \hookrightarrow J^{-1}(2-\delta)$ will be useful to compute initial degrees. 


\begin{setting}\label{setupcurves} In addition to \autoref{sett-vectfield} assume that $X=\C\subset {\mathbb P_k^{n-1}}$ is a reduced equidimensional curve over a perfect field $k$.
Let  $f_1, \ldots, f_{n-2}$ be
forms in $I$ of degrees $\delta_1, \ldots, \delta_{n-2}$
that generate $I$ generically, and let $J$ be the ideal generated by the images in $R$ of the maximal minors of the Jacobian 
matrix of $f_1, \ldots, f_{n-2}$.
Set $\delta=\sum_{j=1}^{n-2}(\delta_j-1)$.
\end{setting}

\begin{remark}{\rm We will see, as a consequence of \autoref{htJ}(a), that the forms $f_1, \ldots, f_{n-2}$ in \autoref{setupcurves} generate $I$ generically if and only if  $\htt J >0$.

If $I$ is a complete intersection, then $f_1, \ldots, f_{n-2}$ can be chosen to be a minimal homogeneous generating sequence of $I,$
in which case $J$ is the full Jacobian ideal of $R$.

If $k$ is infinite and 
$I$ is generated by forms of degrees $\delta_1 \ge\ldots \ge \delta_m$,
then $f_1, \ldots, f_{n-2}$ can be taken to be
$n-2$ general forms of degrees  $\delta_1 \ge\ldots \ge \delta_{n-2}$ in $I$. In this case $f_1, \ldots, f_{n-2}$ also form a regular sequence.
}
\end{remark}

\begin{proposition}\label{3iso}  
Adopt \autoref{setupcurves}. 
\begin{enumerate}[$($a$)$]
\item There exist natural homogeneous $R$-linear map
$$\begin{tikzcd}[row sep={0.3em}] \bigwedge^2 \Omega_k(R) \arrow[two heads]{r} &H\\
 \bigwedge^2 \Omega_k(R) \arrow[two heads]{r} & J(\delta-2)\\
\bigwedge^2 \Omega_k(R) \arrow{r} &\omega_R\, ,
\end{tikzcd}$$
where the first two maps are epimorphisms and all maps are isomorphisms generically. 
\item After factoring out the $R$-torsion of $\bigwedge^2 \Omega_k(R)$ and $H$, or after dualizing into $R$, the first two maps become isomorphisms and the last map becomes an embedding,
\[H/{\rm tor}(H)\cong J(\delta-2) \cong  \wedge^2 \Omega_k(R) /{\rm tor}( \wedge^2 \Omega_k(R) ) \hookrightarrow \omega_R \]
and
\[H^*\cong J^{-1}(2-\delta) \cong  (\wedge^2 \Omega_k(R))^* \hookleftarrow \omega_R^*\, . \]
In particular, \[\Der_k(R)/\m^{-1}\varepsilon\, \hookrightarrow J^{-1}(2-\delta) \, .\]
If $\C$ is smooth and arithmetically Cohen-Macaulay, then the embedding $\omega_R^* \hookrightarrow  (\wedge^2 \Omega_k(R))^*$ is an isomorphism. 
\end{enumerate}
\end{proposition}
\begin{proof} The first map is the second differential onto its image in the 
Koszul complex of the homomorphism $\Omega_k(R) \lto R$ corresponding to the Euler derivation. This map is homogeneous and surjective, and the module $H$ has rank one because the Koszul complex is exact locally on the punctured spectrum of $R$. 

The second map is a direct consequence of the fact that $\Omega_k(R)$ is a module of rank $2$ generated by $n$ elements
and 
$J$ is generated by the maximal minors of the matrix consisting of $n-2$ 
columns of a matrix presenting $\Omega_k(R)$.
Indeed, let  $x_1, \ldots, x_n$ be the images in $R$ of the variables of $S$, 
extend $f_1, \ldots, f_{n-2}$ to a homogeneous generating sequence $f_1, \ldots, f_m$ for $I$,
let $\Theta$ be the image in $R$ of the transpose of the Jacobian matrix of $f_1, \ldots, f_m$,   let $\Theta^{'}$ be the submatrix of $\Theta$ consisting of the first $n-2$ columns  of $\Theta$, and for $1\le i<j\le n$, let  $\Delta_{ij}$ be the  maximal minor of $\Theta^{'}$ with rows $i$ and $j$ deleted.  Notice that $\Theta$ is a homogeneous presentation matrix of $\Omega_k(R)$, that $I_{n-2}(\Theta^{'})=J$, and that $\deg \Delta_{ij}=\delta$.  Since $\Omega_k(R)$ is an $R$-module of rank 2, it follows that $\Theta$ has rank $n-2$.
Now the second  natural map 
$$\begin{tikzcd}\bigwedge^2 \Omega_k(R) \arrow[two heads]{r} & J(\delta-2)
\end{tikzcd}$$
is the homomorphism sending $dx_i \wedge dx_j$ to $(-1)^{i+j} \Delta_{ij}$. This map is well defined because $\Theta$ is a presentation matrix of $\Omega_k(R)$ and has rank $n-2$. The map is obviously homogeneous and surjective. Also notice that $J$ has positive grade in $R$ because the module generated by the columns of
$\Theta'$ has rank $n-2$ as it is generically equal to the syzygy module of $\Omega_k(R)$.

The third map is the canonical class of $R$ over $k$ (see for instance \cites{E, EA, A, L, KW}). This map is homogeneous and it is an isomorphism locally at the regular prime ideals of $R$. 

Since the first two maps are epimorphisms between modules of the same rank, namely one, we see that these maps are also isomorphisms generically. This completes the proof of part (a). Part (b) follows from (a); for the last assertion, we also use \autoref{1.2}. 
\end{proof}

\smallskip

As a first immediate consequence of \autoref{1.2} and \autoref{3iso} we obtain:

\begin{corollary}\label{anticanonicalM}  Adopt \autoref{setupcurves} and assume that $\C$ is smooth and arithmetically Cohen-Macaulay. If $\, \indeg \omega_R^*> a(R)$, then 
$$\Der_k(R)/R \varepsilon \cong \omega_R^*\, .$$
\end{corollary}
\begin{proof} From  \autoref{1.2} and \autoref{3iso}(b) we obtain isomorphisms $L^* \cong H^*\cong \omega_R^*. $ Now again by \autoref{1.2} there is an exact sequence
$$0\lto \Der_k(R)/R\varepsilon \lto \omega_R^* \lto \Ext^2_R(k,R)\, .
$$
Thus the assertion follows  once we have shown that $\Ext^2_R(k,R)$ is concentrated in degrees $\le a(R).$ 

For this we may assume that $k$ is infinite. Since $R$ is Cohen-Macaulay, there exists a regular sequence $x_1, x_2$ consisting of linear forms in $R. $
We have
\[\Ext^2_R(k,R)\cong \Hom_R(k, R/(x_1,x_2))(2) \cong \socle (R/(x_1,x_2)) (2)\, ,\]
and the last module is concentrated in degrees at most $a(R/(x_1,x_2))-2=a(R)$. 
\end{proof}

\smallskip

\begin{corollary}\label{1.-2}  Adopt \autoref{setupcurves}. 
Let $\mu$ be a vector field on ${\mathbb P_k^{n-1}}$ of degree $m$ that leaves $\C$ invariant and whose singular locus
does not contain an irreducible component of $\C$, which means that $\htt \im \mu>0$. 
Then
\[(\im \mu)(m-1) \cong J(\delta-2) \, .\]
\end{corollary}
\begin{proof} The map $\mu$ induces a homogeneous epimorphism of degree $m-1$
$$\begin{tikzcd} H \arrow[two heads]{r}{\mu}  & \im \mu \,. 
 \end{tikzcd} $$
Recall that the $R$-module 
$H$ has rank one.
Since $\grade \im \mu >0$, the vector field $\mu$ induces a homogeneous isomorphism after factoring out the torsion of $H$,
\[ H/{\rm tor}(H)\cong (\im \mu)(m-1)\, .\]
The assertion now follows from \autoref{3iso}(b). 
\end{proof}

\medskip

\section{The invariants}\label{SecIN}

In this section we discuss the invariants that play a role in our estimates. 
\vspace{.1cm}

\noindent
{\it a-invariant.} In many of our bounds on curves, the $a$-invariant replaces Castelnuovo-Mumford regularity if the curve is not arithmetically Cohen-Macaulay. 
The $a$-{\it invariant} of a Noetherian standard graded algebra $R$ over a field is defined as $a(R)=-{\rm indeg}(\omega_R)$. Local duality implies that 
\begin{equation}\label{a-inv-reg}
a(R) \leq {\rm reg} \, R - {\dim}\, R
\end{equation} 
and equality holds if $R$ is Cohen-Macaulay.
\vspace{.1cm}

\noindent
{\it Jacobian ideals.} In this paper, Jacobian ideals will play an important role. To recall
the general definition, let $S=k[x_1, \ldots, x_n]$ be a polynomial ring over a field $k$, 
$W \subset S$ a multiplicative subset, $I \subset W^{-1}S$ an ideal, and $R=(W^{-1}S)/I$. Assume that every minimal prime ideal of $I$ has the same height $g$ and set $D=n-g .$ The {\it Jacobian ideal} of the $k$-algebra $R$ is defined as
$$J_R=J_{R/k}={\rm Fitt}_D(\Omega_k(R))\, .$$
It turns out that
$D=\dim R_{\p}+\trdg_k \kappa(\p)$ for every $\p\in \Spec R$,
where $\kappa(\p)$ denotes the residue field of $\p$, see \autoref{localize}. In particular, $D$ only depends on $k\subset R$, does not change when passing to a nonzero ring of fractions, and coincides with the integers $s$ of \autoref{jak} and $D$ of \autoref{htJ}. Moreover, 
for $V\subset R$ a multiplicative closed subset, one has
$J_{V^{-1}R}=V^{-1}J_{R}\,.$

We will use the following version of the Jacobi criterion.

\begin{theorem}[Jacobi Criterion]\label{jak}
Let $(A,\m, L)$ be a local algebra essentially of finite type over a field $k$, with separable residue field extension $k \subset L\, .$

Then $A$ is regular if and only if 
$\fitt_s\big(\Omega_k(A)\big) =A$ for some $s\leq \dim A +\trdg_k L.$
In this case, the extension $k\subset \Quot(A)$ is separable and $s=\dim A +\trdg_k L.$
\end{theorem}

\smallskip

In our estimates we will also need to use partial Jacobian
ideals as in \autoref{setupcurves}. 
The next results
give Jacobi-like criteria for such ideals.


For a Noetherian ring $R$ and $i\ge 0$ an integer, $\Spec(R)$ is said to be {\it connected in dimension i} if $i< \dim R$ and $\Spec(R)$ cannot be disconnected by removing a closed subset of dimension $<i$. Assume $d=\dim R>0$, then $\Spec(R)$ is connected in dimension $d-1$ if $R$ is a domain with $d<\infty$ or $R$ is an equidimensional catenary local ring satisfying Serre's condition $S_2$ (for the latter case one uses Hartshorne's Connectedness Lemma \cite{HarCON}).

\begin{lemma}\label{GR} Let $T$ be a Noetherian local ring of dimension $d>0$ and assume that $T$ is analytically irreducible or Cohen-Macaulay or, more generally, $\Spec(\widehat{T})$ is connected in dimension $d-1$. Let $\a \subset I$ be ideals and $K=\a : I$. If $I_{\p}=\a_{\p}$ for some $\p\in V(I)$ and $\sqrt{I}\not=\sqrt{\a}$, then
$$\htt (I+K) \le \mu(\a) +1\, .
$$
\end{lemma}
\begin{proof}
We may pass to the completion of $T$ to assume that $T$ is a complete local ring and $\Spec(T)$ is connected in dimension $d-1$. Set $A=T/\a$ and  $s=\mu(\a)$. By Grothendieck's Connectedness Theorem \cite{FOV}*{3.1.7}, $\Spec(A)$ is connected in dimension $d-1-\mu(\a)=d-1-s$. 

On the other hand, our assumptions on $\a$ and $I$ mean that $V(I)\setminus V(I+K)\not= \emptyset$ and  $V(K)\setminus V(I+K)\not= \emptyset$, or equivalently, $V(IA)\setminus V(IA+KA)\not= \emptyset$ and  $V(KA)\setminus V(IA+KA)\not= \emptyset$. As $\Spec(A)=\Spec(IA) \cup \Spec(KA)$, we see that $\Spec(A) \setminus V(IA+KA)$ is disconnected. This can only happen if $\dim T/(I+K)=\dim A/(IA+KA)\ge d-1-s$. It follows that $\htt (I+K)\le s+1$. 
\end{proof}

\vspace{.1cm}

\begin{theorem}\label{htJ}
Let $(T, \m, L)$ be a regular local ring essentially of finite type over a perfect field $k$. Let $I$ be an ideal of height $g$ and $\a=(f_1, \ldots, f_g)\subset I$. Write $A=T/\a$ and $R=T/I$ and assume $R$ is equidimensional of dimension $\ge 2$. Set $D=\dim R+ \trdeg_k L$ and consider the Jacobian-like ideal $J=\Fitt_D(R\otimes_A \Omega_k(A)) \subset R\, .$
\begin{enumerate}[$($a$)$]
\item $\htt J \ge 1$ if and only if $I_{\p}=\a_{\p}$ for every minimal prime $\p$ of $I$ and $R$ satisfies Serre's condition $R_0$. 
\item If $\htt J \ge 2$ then $I=\a$ is a complete intersection. 
\item $\htt J \ge i$  for some $i\ge 2$ if and only if $I=\a$ is a complete intersection and $R$ satisfies Serre's condition $R_{i-1}$. 
\end{enumerate}
\end{theorem}
\begin{proof}
We first prove that if $\p$ is a prime ideal in $V(I)$ with residue field $\kappa$, then 
\begin{equation}\label{localize}
\dim R_{\p} +\trdeg_k \kappa =D\, .
\end{equation}

Since $T$ is the localization of a finitely generated $k$-algebra, we can write $T=T'_{\m'}$, where $T'$ is a finitely generated $k$-subalgebra of $T$ and $\m'=\m\cap T'$. Notice $T'$ is a domain. Set $\p'=\p\cap T'$, $I'=I\cap T'$, and $R'=T'/I'$. The minimal primes of $I'$ are contracted from minimal primes of $I$ and hence they all have the same height $g$. Therefore $R'$ is equidimensional. Now
$$\dim R_{\p}+\trdeg_k \kappa=\htt \p'R'+ \dim R'/\p'R'=\dim R'=\htt \m'R'+ \dim R'/\m'R'=\dim R +\trdeg_k L \, ,
$$
as claimed. 

It remains to prove that if $\htt J\ge 1$, then $I_{\p}=\a_{\p}$ for every minimal prime $\p$ of $I$ and if $\htt J\ge 2$, then $I=\a$. The rest follows from \autoref{jak} and \autoref{localize}. 

We first show that if $\p \in V(I)$ and $J_{\p}=R_{\p}$, then $I_{\p}=\a_{\p}$. We wish to apply \autoref{jak} to the ring $A_{\p}$ with $s:=D$. Notice that $$s=D=\dim R_{\p}+\trdeg_k \kappa \le \dim A_{\p} +\trdeg_k \kappa\, $$ 
according to \autoref{localize} and that $\Fitt_s(\Omega(A_{\p}))=A_{\p}$ because $I\subset \p$. Now \autoref{jak} implies that $A_{\p}$ is regular and $\dim A_{\p}=\dim R_{\p}$. 
As $A_{\p}$ is a domain mapping onto $R_{\p}$, we conclude that $\a_{\p}=I_{\p}$ as asserted. 

Thus we have proven that if $\htt J\ge 1$, then $\a_{\p}=I_{\p}$ for every minimal prime $\p$ of $I$. On the other hand, if $\htt J\ge 2$ we conclude that $\a_{\p}=I_{\p}$ for every $\p \in V(I)$ with $\dim T_{\p}\le g+1$. So for $K:=\a : I$, we have $\htt (I+K) \ge g+2> \mu(\a)+1$. Also $\a_{\p}=I_{\p}$ for some $\p\in V(I)$, in fact for every minimal prime $\p$ of $I$. Therefore \autoref{GR} shows that $\sqrt{\a}=\sqrt{I}$. In particular, $\a$ is a complete intersection. Thus every associated primes $\p$ of $\a$ is a minimal prime of $\a$, hence a minimal prime of $I$ because $\sqrt{\a}=\sqrt{I}$. Therefore, $\a_{\p}=I_{\p}$. Since this holds for every associated prime of $\a$, we obtain $\a=I$. 
\end{proof}

\vspace{.1cm}
\noindent
{\it{Tjurina number.}} The estimates for plane curves in \cites {dPW, EK} use the sum of the Tjurina numbers at the singular points. To allow for curves in projective spaces of arbitrary dimension, we replace the sum of the Tjurina numbers by the degree of the singular locus endowed with the scheme structure given by the Jacobian ideal, which is the multiplicity of the homogenous coordinate ring of the curve modulo its Jacobian ideal. If the curve is locally a complete intersection, as is the case for any plane curve, then the sum of the Tjurina numbers and the degree of the singular locus coincide, see \autoref{basicC}. 


Let $A$ be a local ring essentially of finite type over a field $k$. By ${\rm T}^1(A/k)$ we denote the first cotangent cohomology of the $k$-algebra $A$. If $k$ is perfect and
$A$ is reduced, then ${\rm T}^1(A/k)\cong \Ext^1_A(\Omega_k(A), A)\, .$ The module ${\rm T}^1(A/k)$ has finite length whenever $k$ is perfect and $A$ has an 
isolated singularity; this length is called the {\it Tjurina number} of $A$ and denoted by $\tau(A)$. If the residue field extension is trivial, then $\tau(A)$
is the embedding dimension of the formal moduli space of $A$, the parameter space of the versal deformation of the $k$-algebra $\widehat{A}$. The {\it total
Tjurina number} of a reduced curve $\mathcal C \subset \mathbb P^{n-1}_k$ over a perfect field is defined as $\tau(\mathcal C)=\sum\limits_{p\in {\rm Sing}(\C)} \tau(\mathcal O_{\C, p})$.

\begin{lemma} \label{basicL} Let $k$ be a perfect field, $X\subset \mathbb P^{n-1}_k$ be a reduced and equidimensional subscheme, and $Y\subset X$ be a subvariety. Let $R$ be the homogenous coordinate ring of $X$, let $\p\subset R$ be the prime ideal defining $Y$, and write $T=R_{\p}$ and $A=\mathcal O_{X,Y}\, .$
\begin{enumerate}[$($a$)$]
\item $T\cong A(x)$ for every $x\in R_1 \setminus \p\, ;\, $ any such $x$ is transcendental over $A$.
\item $\Omega_k(T)\cong (\Omega_k(A)\otimes_A T) \oplus Tdx\, ,$ where $T \, dx \cong T$.
\item $J_{T/k}=J_{A/k} T$ and ${\rm T}^1(T/k)\cong {\rm T}^1(A/k) \otimes_A  T$.
\end{enumerate}
\end{lemma}
\begin{proof} Part (a) is well known, part (b) is an immediate consequence of (a), and part (c) follows from (b). 
\end{proof}

\begin{proposition}\label{basicP} Let $k$ be a perfect field and $A$ be a local $k$-algebra essentially of finite type with algebraic residue field extension. If $A$ is a reduced complete intersection of dimension one, then $$\tau(A)=\lambda(A/J_A) .$$
\end{proposition}
\begin{proof} Notice that ${\rm projdim}_A \Omega_k(A) \le 1$ and $\rank_A \Omega_k(A)=\dim A=1$. Thus \cite{Storch}*{Satz} implies $\lambda(A/J_A)= \lambda({\rm tor}(\Omega_k(A)))$. Since
$A$ is Gorenstein, local duality gives $\lambda({\rm tor}(\Omega_k(A)))=\lambda(\Ext^1_A(\Omega_k(A),A))$. As $ \Ext^1_A(\Omega_k(A),A)\cong {\rm T}^1(A/k),$ the assertion now follows.
\end{proof}

The next corollary expresses the total Tjurina number of a local complete intersection curve, which is defined in terms of local invariants of the singular points, as
a global invariant, the degree of 
the singular scheme of the curve, which
can be computed without knowing the singularities. It is this global 
invariant that replaces the global Tjurina number in our estimates when the curves need not be a local complete intersection.

\begin{corollary}\label{basicC} Let $k$ be a perfect field and $\C \subset \mathbb P^{n-1}_k$ be a singular reduced curve that is locally a complete intersection. Write $R$ for the homogeneous coordinate ring of $\C$. Then
$$\tau(\C)=e(R/J_R)\, .$$
\end{corollary}
\begin{proof} Write $\m$ for the homogenous maximal ideal of $R$.  In this case we have
\begin{eqnarray*}
\tau(\C)&=&  \hspace{ .2 cm}  \sum_{p \in {\rm Sing}(\C)} \lambda_{{\mathcal O_{\C, p}}} ({\mathcal O_{\C, p}}/J_{{\mathcal O_{\C, p}}})  \hspace{.9 cm}  \text{by \autoref{basicP}}\\
&=& \sum_{\p \in V(J_R)\setminus \{\m\}} \lambda_{R_{\p}} (R_{\p}/J_{R_{\p}})  \hspace{1.4 cm}  \text{by \autoref{basicL}(c)}\\
&=& \hspace{ .2 cm} e(R/J_R)  \hspace{3.85 cm}  \text{by the associativity formula for multiplicity.}
\end{eqnarray*}
\end{proof}

\noindent
{\it Loewy multiplicity.} 
Besides the multiplicity of the homogeneous coordinate ring of a curve modulo its Jacobian ideal, we will 
also consider what we call the Loewy multiplicity, which is defined by replacing length by
Loewy length in the associativity formula for multiplicity. 

The {\it Loewy length} of a module $M$ of finite length over a local ring $(A, \mathfrak m)$ is the smallest integer $s\ge 0$ so that
$\m^sM=0$. The Loewy length satisfies the inequality $\ell\ell(M) \leq \lambda(M)$, which is an equality if and only if $M$ and $\m M$ are cyclic if
and only if every submodule of $M$ is cyclic.
We will use a strengthening of this inequality:

\begin{proposition}\label{lmult} Let $A$ be the local ring of a point on a reduced plane curve over a perfect field and write $e=e(A)$. Then 
$$  \ell\ell(A/J_A)\le \lambda(A/J_A)-{e-1\choose 2}\, .
$$
\end{proposition}
\begin{proof} We may assume that $A=S/(f)\, ,$ where $S=k[x,y]_{(x,y)}\, .$ Let $\n$ be the maximal ideal of $S$ and $\m$ be the maximal ideal of $B:=A/J_A\, .$ We have $f\in \n^e$ and so $(\frac{\partial f}{\partial x}, \frac{\partial f}{\partial y})\subset \n^{e-1}\, .$ It follows that $\m^i/\m^{i+1}\cong \n^i/\n^{i+1}\cong k^{i+1}$ for $i\le e-2\, .$ Therefore 
$$\lambda(A/J_A) -\ell\ell(A/J_A)\ge \sum_{i=0}^{e-2} (\lambda(\m^i/\m^{i+1}) -1) \ge {e-1\choose 2}\, .
$$
\end{proof}

Let $M$ be a finitely generated module over a Noetherian ring $R$, where either $R$ is local or else $M$ is graded and $R$ is positively 
graded over an Artinian local ring. Recall that 
$$e(M)=\sum \lambda(M_{\p}) \cdot e(R/\p)\, ,$$ 
where $\p$ ranges over all
prime ideals of maximal dimension in ${\rm Supp}(M)$. Analogously, we define the {\it Loewy multiplicity} of
$M$ as
$${\rm Lmult}(M):=\sum \ell \ell (M_{\p}) \cdot e(R/\p)\, .$$ 
Clearly, ${\rm Lmult}(M) \leq e(M)$ and equality holds if and only if for every $\p$ as above, every $R_{\\p}$-submodule of $M_{\p}$ is cyclic.

\vspace{.2cm}

\noindent
{\it Singularity degree and genus.} If $A$ is an analytically unramified Noetherian local ring, with integral closure $\overline A$, then the $A$-module $\overline A/A$ has finite length if
and only if $A$ is normal locally on the punctured spectrum. In this case, $\sigma(A):=\lambda(\overline A/A)$ is called the {\it singularity degree} of $A$. The singularity 
degree of a reduced curve $\mathcal C \subset \mathbb P^{n-1}_k$ is defined as
$$\sigma(\mathcal C)=\sum_{p\in \Sing(\C)} \sigma(\mathcal O_{\mathcal C, p}) \,.
$$
An argument as in the proof of \autoref{basicC} shows that $\sigma(\mathcal C)=e(\overline R/R)$ if $\C$ is singular, where $R$ is the homogeneous coordinate ring of $\mathcal C$.

The singularity degree of a curve is closely related to its arithmetic and geometric genus. Let $X \subset \mathbb P_k^{n-1}$ be a reduced subscheme over
a field $k$, with homogeneous coordinate ring $R$. Let $\m$ be the homogeneous maximal ideal and $\p _1, \ldots, \p_s$ the minimal primes of $R$, and
write $R_i=R/\p _i$. Each $\overline{R_i}$ are positively graded algebra over a finite field extension $k_i$ of $k$, and $\overline R$ is a positively
graded algebra over the Artinian ring $K:=\times k_i$. A suitable Veronese subring of $\overline R$ is a standard graded algebra over $K$ and is the
homogeneous coordinate ring of the normalization of $X$ embedded into a projective space over $K$.

We describe, in passing, an embedding of the normalization into a projective space over the field $k$, when $k$ is algebraically closed. In this case, $k_i=k$
and we may define the natural projections
${\pi}_i: \overline{R_i} \twoheadrightarrow k$.
We consider the fiber product $\widetilde R := \{ (a_i) \in \times \, \overline{R_i} \ | \ {\pi}_i(a_i)={\pi}_j(a_j) \ \forall  \ i \neq j \}$. One 
has $R \subset \widetilde R \subset \overline R$ and $\m \overline R \subset \widetilde R$. The ring $\widetilde R$ is a positively graded algebra 
over the field $k$, and one sees that a Veronese subring of $\widetilde R$ is the homogeneous coordinate ring of the normalization of $X$ embedded into
 a projective space over $k$. If $X$ is equidimensional, then $\omega_{\widetilde R} \cong \omega_{\overline R}:= \times \omega_{\overline{R_i}}$. Also
 notice that the degree zero component of the canonical module is unchanged by passing to 
a Veronese subring.

Now let $\mathcal C \subset \mathbb P_k^{n-1}$ be a reduced curve over a field, with homogeneous coordinate ring $R$.
The {\it arithmetic genus}  $p_a$ of
$\mathcal C$ is 1 minus the constant term of the Hilbert polynomial of $R$. If $k$ is algebraically closed, the {\it geometric genus}  $p_g$ of
$\mathcal C$ can be defined as $\dim_k [\omega_{\overline R}]_0=\dim_k  [\omega_{\widetilde R}]_0$.
 
 \smallskip
 



The following lemma and proposition are well known, we give a proof for the convenience of the reader.

\begin{lemma}\label{pa}
Let $\C\subset \mathbb P_k^n$ be a reduced arithmetically Cohen-Macaulay curve over a field $k$, with arithmetic genus $p_a$ and homogeneous coordinate ring $R$. Then 
$$p_a= \dim_k [\omega_R]_0\, .
$$
\end{lemma}
\begin{proof}
Let $\m$ be the homogeneous maximal ideal of $R,$ and let $h$ and $p$ denote the Hilbert function and the Hilbert polynomial of $R$, respectively. One has
$$p_a=h(0)-p(0)=\dim_k [H^2_{\m}(R)]_0=\dim_k [\omega_R]_0\, ,
$$
where the second equality follows from the Grothendieck-Serre formula and the third equality is a consequence of local duality. \end{proof}

\begin{proposition}\label{genus2} Let $k$ be an algebraically closed field  and  $\C\subset \mathbb P_k^{n-1}$ be a reduced curve with $s$ irreducible components, arithmetic genus $p_a$, geometric genus $p_g, $ and singularity degree $\sigma$. One has 
 $$ p_a-p_g=\sigma -s+1\, .
 $$
  \end{proposition}
 \begin{proof} Let $R$ be the homogeneous coordinate ring of $\C$ and $\overline{R}$ its integral closure.  We may assume that $R$ and $\overline{R}$ are standard graded after passing to Veronese subrings; this does not change the local rings of $\C$, the constant term of the Hilbert polynomial of $R$, which is $1-p_a$, and the degree zero component of $\omega_{\overline{R}}\, .$ 
So $p_g=\dim_k [\omega_{\overline{R}}]_0\, .$
 
 We compare the constant terms of the Hilbert  polynomials of the graded $R$-modules in the  exact sequence 
$$0\lto R\lto \overline{R} \lto \overline{R}/R \lto 0\, .
$$
The constant term of the Hilbert polynomial of $R$ is $1-p_a$. One has $\overline{R}=\times_{i=1}^{s} \overline{R_i}$, where $\overline{R_i}$ are standard graded Cohen-Macaulay algebras over $k$. Applying \autoref{pa} we see that the constant term of the Hilbert polynomial of $\overline{R}$ is 
$$\sum_{i=1}^{s} (1-\dim_k[\omega_{\overline{R_i}}]_0)=s-\dim_k[\omega_{\overline{R}}]_0=s-p_g\, .$$ Finally, the constant term of the Hilbert polynomial of $\overline{R}/R$ is 
$\sigma$. Now the additivity of the Hilbert polynomial in short exact sequences implies that 
 $$s-p_g=1-p_a+\sigma,$$
 as claimed. 
  \end{proof}

\medskip

\section{Lower bounds for hypersurfaces and curves with at most planar singularities}\label{SecLL}

One of the main results of this section is an estimate for the degree of a vector field in terms of the $a$-invariant of $R$,
the number of singular points, and the Loewy multiplicity modulo the Jacobian ideal, see \autoref{curves}.
It says that if $\mathcal C$ has only plane singularities, then

\vspace{.26 cm}
\centerline{
$
\findeg (\Der_k(R)/\m^{-1}\varepsilon )\geq a(R) +1 + | {\rm Sing}(\C) | - {\rm Lmult}(R/ {\rm Jac}(R)) \, .
$}
\vspace{.26cm}

\noindent
 If $\mathcal C$ has only ordinary nodes as singularities, then 
$|{\rm Sing}(\C) | - {\rm Lmult}(R/ {\rm Jac}(R))=0$ and we obtain
the inequality $\indeg (\Der_k(R)/\m^{-1}\varepsilon )\geq a(R) +1$, which
we prove to be an equality when $\mathcal C$ is arithmetically Gorenstein, see \autoref{Equality}.
The case of ordinary nodes  had
been treated before with the additional assumption that, first, $\C$ is a plane curve \cite{CL},
then, $\C$ is a complete intersection \cite{CCF}, and, finally, $\C$
is arithmetically Cohen-Macaulay \cites{E, EK2002}.

\begin{proposition}\label{syz} 
Let $S=k[x_1,\ldots,x_n]$ be a polynomial ring in $n$ variables over a field and let $\m_S$ denote its maximal homogeneous ideal.   
 Let $f$ be a homogeneous polynomial of degree $d$ and assume that $d$ is not a multiple of the characteristic. Denote the partial derivatives of $f$ by $f_1, \ldots, f_n$ and 
 let $\B$, $\Z$, and $\H$ be the modules of first boundaries, cycles, and homology of the Koszul complex of $f_1, \ldots, f_n$. Write $R=S/(f)$. 
\begin{enumerate}[$($a$)$]
 \item There are natural  epimorphisms of homogeneous $S$-modules
$$\begin{tikzcd}
\Z (d) \arrow[two heads]{r}
  &  \Der_k(R)/R\varepsilon  \arrow[two heads]{r} & (\Z/\m_S \B)(d) \end{tikzcd};$$
in particular,
$$ k \otimes_S \Z(d)\cong k \otimes_R  \Der_k(R)/R\varepsilon$$
 \item $$\mu(\Der_k(R)/R\varepsilon) =\beta^2_S\left(\frac{S}{(f_1, \ldots, f_n)}\right)\le (d-1)^n$$
\item $$\indeg (\Der_k(R)/R\varepsilon)  = \indeg (\Z)-d=\min \{  d-2, \, \indeg (\H)-d\} .$$
\end{enumerate}
\end{proposition}
\begin{proof} 
Dualizing the exact sequence
$$\begin{tikzcd} R(-d) \arrow{rr}{[f_1,\ldots,f_n]^t} &&R^n(-1) \arrow{r} &\Omega_k(R)\arrow{r}&0
 \end{tikzcd} $$
into $R$, we obtain
\begin{equation}\label{hamid1}\begin{tikzcd} 0\arrow{r}&\Der_k(R) \arrow{r} &\oplus_{i=1}^n R \, \partial/\partial x_i  \arrow{rr}{[f_1,\ldots,f_n]} && R(d)
 \end{tikzcd} .\end{equation}
There is a homogeneous map $$ \varphi: \Z(d) \lto \Der_k(R)$$ induced by the diagram
\begin{equation}\label{Hamid}\begin{tikzcd} 
0\arrow{r}&\Der_k(R) \arrow{r} &\oplus_{i=1}^n R \, \partial/\partial x_i  \arrow{rr}{[f_1,\ldots,f_n]} && R(d)\\
0\arrow{r}& \Z(d) \arrow{r} \arrow[u, "\varphi"]& \oplus_{i=1}^n Se_i\arrow{rr}{[f_1,\ldots,f_n]}\arrow[u, twoheadrightarrow]  &&S (d) \arrow[u, twoheadrightarrow]
 \end{tikzcd} \, 
 \end{equation}
 where the homogeneous basis elements $e_i$ have degree $-1$. 
 
We claim that the composition $\psi$ 
$$\begin{tikzcd}
\Z(d) \arrow{r}
  &\Der_k(R)  \arrow[two heads]{r} & \Der_k(R)/R\varepsilon 
    \end{tikzcd}
$$
 is surjective. Let $g=\sum b_i\, \partial/\partial x_i\in \Der_k(R)$. According to \autoref{hamid1}, $\sum b_if_i=cf$ for some $c\in S$. Using the Euler relation $f=\frac{1}{d} (x_1f_1+\ldots+x_nf_n)$, we obtain $\sum (b_i-\frac{c}{d}x_i)f_i=0$. Therefore $\sum (b_i-\frac{c}{d}x_i) e_i \in \Z(d)$  and 
 $\varphi (\sum (b_i-\frac{c}{d}x_i)e_i) =g -\frac{c}{d} \varepsilon.$

To show part (a) it suffices to show  that $\ker \, \psi \subset \m_S \B$. The diagram \autoref{Hamid} shows that \begin{equation}\label{hamid2}\ker \psi= (fS^n+ S\sum_{i=1}^n x_ie_i) \cap \Z(d)\, .\end{equation}
 
Let
$z=\sum_{i=1}^n\alpha_i e_i \in \ker \, \psi.$ According to \autoref{hamid2} there exists $a_1, \ldots, a_n, b$ in $S$ so that 
\begin{equation}\label{hamid3}
z=\sum_{i=1}^n a_i f e_i + b \sum_{i=1}^n x_i e_i \, .
\end{equation}
Since $z$ is a syzygy of $f_1, \ldots, f_n$ we obtain 
$$0 =\sum_{i=1}^n \alpha_i f_i = \sum_{i=1}^n a_i f f_i +  b \sum_{i=1}^n  x_i f_i =  \sum_{i=1}^n a_i f f_i +  d b  f\, .
$$
Dividing by $f$ it follows that
$$b=-\frac{1}{d}\sum_{i=1}^n a_i  f_i \, .
$$
Substituting the above expression of $b$ into \autoref{hamid3} and using the Euler relation we see that
$$\alpha_i=\sum_{j=1}^n \frac{1}{d} (a_ix_j-a_jx_i) f_j \, .
$$
Since the $n \times n$ matrix $(a_ix_j-a_jx_i)$ is alternating and has entries in $\m_S$ it follows that $z\in \m_S \B (d)$. 

The equality in (b) and the first equality in (c) are direct consequences of the isomorphism in (a). The inequality in (b) follows from \cite{DaHS}*{3.10}. For the second equality in (c), notice that  $\indeg (\Z)-d\le \indeg(\B)-d =d-2$ and $\indeg (\Z)\le \indeg(\H)$ and that either $\indeg (\Z)=\indeg(\B)$ or else  $\indeg (\Z)=\indeg(\H)$.
\end{proof}

The next theorem is a generalization of \cite{EK2003}*{Theorem 2.5} from plane curves to hypersurfaces.

\begin{theorem}[The hypersurface case]\label{hyper} Let $k$ be a perfect field and $X\subset {\mathbb P_k^{n-1}}$ be a reduced hypersurface of degree $d$. Let $R$ be the homogenous coordinate ring of $X$ and $J_R$ the Jacobian ideal. Assume $n\ge 3$ and that $d$ is not a multiple of the characteristic. If $X$ has at most isolated singularities, then 
$$\indeg (\Der_k(R)/R\varepsilon)  =  \begin{cases} \min \{  d-2, (n-1)(d-2) -a(R/J_R)-2\}  \\ 
d-2 & \mbox{if } X \mbox{ is smooth } \, .\end{cases}
$$
\end{theorem}
\begin{proof}
We may assume that $k$ is infinite. Let $S=k[x_1,\ldots,x_n]$ be a polynomial ring in $n$ variables, and write $R=S/(f)$ with $f$ a form of degree $d$. Denote the partial derivatives of $f$ by $f_1, \ldots, f_n$ and set $\J=(f_1, \ldots, f_n)$. Notice that $J_R$ is the image of $\J$  in $R$. Since $d$ is a unit in $k$, the Euler relation shows that $f\in \J$. Therefore $\htt \J\ge n-1$, and after a linear change of variables we can assume that $f_1, \ldots,  f_{n-1}$ form a regular sequence.  

Let $\H$ be the first homology module of the Koszul complex of $f_1, \ldots, f_n$.  If $X$ is smooth, then $\H=0$ and the claim follows from \autoref{syz}(c). Otherwise  $\htt \J= n-1$ and  $$\H(d-1)=\frac{(f_1, \ldots, f_{n-1}):_S \J}{(f_1, \ldots, f_{n-1})},$$ as $f_1, \ldots,  f_{n-1}$ form a regular sequence. Now\begin{eqnarray*} \frac{(f_1, \ldots,  f_{n-1}):_S \J}{(f_1, \ldots,  f_{n-1})} \cong &\omega_{S/\J} (n-(n-1)(d-1))\\
  =&\omega_{R/J_R} (1-(n-1)(d-2))
\end{eqnarray*}
and the theorem follows again by \autoref{syz}(c)  because $\indeg \omega_{R/J_R}=-a(R/J_R)$. 
 \end{proof}

The connection between the degree of vector fields and the degree of syzygies of Jacobian ideals, which is used in the previous proof, was already observed in \cite{E}*{Remark 9}. 
The use 
of the conductor and the integral closure in the proof of \autoref{PC} below 
was inspired by \cite{EU}*{proof of Corollary 5.1}.



\begin{lemma}\label{productcm1} Let $R$ be a standard graded Noetherian algebra over a field and $\, \mathfrak{b}_1,\ldots,\mathfrak{b}_t$ homogeneous ideals 
of $R$ so that the rings $R/\mathfrak b_i$ have dimension $1.$ Then $$ a(R/{\mathfrak b}_1\cdot \ldots \cdot {\mathfrak b}_t)\leq \sum_{i=1}^t a(R/{\mathfrak b}_i)+2t-2.$$ 
\end{lemma}

\begin{proof} Let $S$ be a standard graded polynomial ring over the ground field of $R,$
mapping homogeneously onto $R,$ and let 
$\widetilde{\mathfrak b}_i$ denote the preimage of $\mathfrak b_i$ in $S$. 
Since
$\, S/ \tilde{ \mathfrak b}_1 \cdot \ldots \cdot \tilde{\mathfrak b_t} \longrightarrow 
R/{\mathfrak b}_1\cdot \ldots \cdot {\mathfrak b}_t \,$ is a surjection of rings having the same dimension, $a( R/{\mathfrak b}_1\cdot \ldots \cdot {\mathfrak b}_t)\le a(S/ \tilde{ \mathfrak b}_1 \cdot \ldots \cdot \tilde{\mathfrak b_t})$. Hence it suffices to prove our assertion for the case that $R$ is a polynomial ring, which we now assume.

If $\mathfrak b^{\text{unm}}$ denotes the unmixed part of a homogeneous ideal $\mathfrak b$ 
of the polynomial ring $R$ such that  $R/\mathfrak b$ has dimension $1$, then $R/\mathfrak b^{\text{unm}}$ is Cohen-Macaulay and
$$
a(R/\mathfrak b )=a(R/\mathfrak b^{\text{unm}} )=\reg (R/\mathfrak b^{\text{unm}})-1=\reg (\mathfrak b^{\text{unm}})-2 \leq \reg (\mathfrak b )-2.
$$
According to the first equality, we can further assume that every ${\mathfrak b}_i$ is unmixed.

In that case, the inequality $\reg  ({\mathfrak b}_1 \cdot \ldots \cdot {\mathfrak b_t} )\leq \sum _{i=1}^t\reg( \mathfrak b_i)=2t+\sum _{i=1}^t a( R/\mathfrak b_i)$ proved in  \cite{Sidman}*{Theorem 1.8}  together with the 
inequality $a  (R/{\mathfrak b}_1 \cdot \ldots \cdot {\mathfrak b_t} )\leq \reg  ({\mathfrak b}_1 \cdot \ldots \cdot {\mathfrak b_t} )-2$ completes the proof.
\end{proof}


\begin{proposition}[The plane curve case]\label{PC} Let $k$ be an algebraically closed field and $\C\subset {\mathbb P_k^2}$ be a reduced curve of degree $d$. Let $R$ be the homogenous coordinate ring of $\C$ and $J_R$ be the Jacobian ideal. Assume that $d$ is not a multiple of the characteristic. One has
\begin{eqnarray*}
\indeg (\Der_k(R)/R\varepsilon)  &=& \begin{cases} \min\{d-2, 2d-6-a(R/J_R)\}  \\ 
d-2  \hspace{9.1cm}  \mbox{if } \C \mbox{ is smooth } \end{cases}\\
&\ge& \begin{cases} d-3 +|\Sing(\C)|-\Lmult(R/J_R) \\ 
d-2 +|\Sing(\C)|-\Lmult(R/J_R) \hspace{4.3cm} \mbox{if } \C \mbox{ is irreducible \,} \\
d-2 \hspace{9cm}  \mbox{if } \C \mbox{ is smooth } \end{cases}\\
&\ge& \begin{cases} d-3 +|\Sing(\C)|+\sum\limits_{p\in {\rm Sing}(\C)} {e(\mathcal O_{\C, p})-1 \choose 2}-\tau(\C) \\ 
d-2 +|\Sing(\C)|+\sum\limits_{p\in {\rm Sing}(\C)} {e(\mathcal O_{\C, p})-1 \choose 2}-\tau(\C) \hspace{2.2cm}  \mbox{if } \C \mbox{ is irreducible } \\
d-2 \hspace{9cm}  \mbox{if } \C \mbox{ is smooth } \, .\end{cases}\\
\end{eqnarray*}

\end{proposition}
\begin{proof} The equality is a special case of \autoref{hyper}, and the second inequality is a consequence of the bound $\Lmult(R/J_R)\le \tau(\C)-\sum\limits_{p\in {\rm Sing}(\C)} {e(\mathcal O_{\C, p})-1 \choose 2},$ which follows from \autoref{basicC} and \autoref{lmult}. 

To prove the first inequality, we estimate $a(R/J_R)$ when $\C$ is singular. Set $J=J_R.$
Let $\p_1, \ldots, \p_t$ be the minimal primes of $J$, which are necessarily of height one; they correspond to the singular points of $\C$ and therefore $t=|\Sing(\C)|$.  Write $s_i=\ell\ell(R_{\p_i}/J_{\p_i})$ for the Loewy length of  $R_{\p_i}/J_{\p_i}$. Let $\f=R:_R \overline{R}$ denote the conductor of $R$ and notice that $\f \subset \sqrt{J}$.  One has 
\[J^{\rm unm}\supset \cap \ \p_i^{s_i} \supset \sqrt{J} \cdot \p_1^{s_1-1} \cdots \ \p_t^{s_t-1}\supset \f \cdot \p_1^{s_1-1} \cdots \ \p_t^{s_t-1}\, .\]
It follows that
\begin{eqnarray*}
a(R/J)=a(R/J^{\rm unm})& \le& a(R/  \f \cdot \p_1^{s_1-1} \cdots \ \p_t^{s_t-1})\\
&\le&  a(R/ \f) + \sum_{i=1}^{t}(s_i-1) \quad \mbox{by \autoref{productcm1} } \\
&=& a(R/\f)+\Lmult(R/J) -|\Sing( \C)|\, .
\end{eqnarray*}

To estimate $a(R/\f)$ we dualize the short exact sequence 
$$0\to \f\to R\to R/\f\to 0$$ 
into $\omega_R \cong R(d-3)$.  As $\text{Hom}_R(\f,R)=\overline{R}$,  we obtain $\omega_{R/\f}=(\overline{R}/R)(d-3)$. 
It follows that $a(R/\f)=-\indeg (\overline{R}/R)+d-3.$ Therefore
$$a(R/J) \leq -\indeg (\overline{R}/R)+d-3 +\Lmult(R/J) -|\Sing( \C)|\, .$$

Finally,  we have $\indeg( \overline{R}/R) \ge 0$ because $R$ is reduced, and $\indeg( \overline{R}/R) \ge 1$ if $R$ is a domain since $k$ is algebraically closed.
\end{proof}

In order to deal with subschemes that are not necessarily arithmetically Gorenstein, we introduce the following notion:

\begin{definition}\label{1.4} Let $R$ be a standard graded algebra over a field. We say that $R$ has the {\it generalized Cayley-Bacharach property} if $\findeg \omega_R=\indeg \omega_R$. 
\end{definition}


Examples of rings with the generalized Cayley-Bacharach property are domains, Gorenstein rings, and more generally level rings. If the ground field is algebraically closed and $R$ is reduced and one-dimensional, then $R$ has the generalized Cayley-Bacharach property if and only if the corresponding set of points in projective space has the Cayley-Bacharach property in the usual sense (see \cite{GKR}). 

\ms

The next theorem is inspired by work of Esteves \cite{E}*{Theorem 17}. We do not require arithmetic Cohen-Macaulayness as in \cite{E} and our proof is short and 
elementary. We will say that a ring extension $A \subset R$ is {\it birational} if $R$ is a torsionfree $A$-module and the induced map  $\Quot(A) \lto \Quot(R)$ is 
an isomorphism. The following fact, which is a special case of  \cite{UV}*{Proposition 5.2}, will be used frequently.

\begin{lemma}\label{UV} Let $k$ be an infinite perfect field and let $R$ be a reduced and equidimensional $k$-algebra of dimension $D$ generated by $y_1, \ldots, y_n.$ If 
$A$ is the $k$-subalgebra generated by $D+1$ general $k$-linear combinations of $y_1, \ldots, y_n,$ then $A \subset R$ is a finite and birational extension and the induced map  $\Quot(A) \lto \Quot(A)\otimes_A R$ is an isomorphism.
\end{lemma} 


\vspace{.02cm}

\begin{theorem}\label{GP} Let $k$ be perfect field and let $A\subset R$ be a finite and birational homogeneous extension of standard graded $k$-algebras. Assume that $R$ is reduced and equidimensional of dimension at least two, with maximal homogenous ideal $\m$, and that $A$ is Gorenstein. 

One has
\[ \findeg (\Der_k(R)/\m^{-1} \varepsilon_R) \ge \indeg ( \Der_k(A)/A \varepsilon_A) -a(A)+a(R)\]
and 
\[\findeg (\Der_k(A)/A \varepsilon_A) \ge \indeg (\Der_k(R)/\m^{-1}\varepsilon_R) -a(A)+a(R)\, .
\]

If in addition $R$ has the generalized Cayley-Bacharach property, then
\[ \findeg (\Der_k(R)/\m^{-1} \varepsilon_R ) \ge \findeg (\Der_k(A)/A\varepsilon_A) -a(A)+a(R)\]
\[ \indeg (\Der_k(R)/\m^{-1} \varepsilon_R) \ge \indeg (\Der_k(A)/A\varepsilon_A) -a(A)+a(R)\]
and 
\[\findeg (\Der_k(A)/A\varepsilon_A) \ge \findeg ( \Der_k(R)/\m^{-1}\varepsilon_R) -a(A)+a(R)
\]
\[\indeg (\Der_k(A)/A\varepsilon_A) \ge \indeg ( \Der_k(R)/\m^{-1}\varepsilon_R) -a(A)+a(R)\, .
\]
\end{theorem}

\bs
Notice that the inverse of the maximal homogeneous ideal of $A$ is equal to $A$ since $\depth  A\ge 2$. 

\begin{proof} We consider the conductor $$\f=A:_{A}R\cong \Hom_A(R,A)\cong \Hom_A(R,\omega_A)(-a(A))\cong \omega_R (-a(A))\, .$$
Notice that $\indeg \f=a(A)-a(R)$. In addition,  $\f$ annihilates the $R$-module $\Omega_A(R)$; indeed, if $d: R \lto \Omega_A(R)$ denotes the universal $A$-derivation of $R$, then for any $a\in \f$ and $r\in R$, 
\[a\, d (r)=d(ar)-r\, d(a)=0\]
since both $ar$ and $a$ are in $A$. 

In the exact sequence 
\[ R \otimes_A \Omega_k(A) \stackrel{\alpha}{\lto} \Omega_k(R) \lto \Omega_A(R) \lto 0\, ,
\]
$\ker \alpha$ and $\Omega_A(R)$ are $R$-torsion modules because the extension $A \subset R$ is birational. 
Thus this sequence induces an exact sequence 
$$
\begin{tikzcd}[row sep={0.9em}]  
 0 \arrow{r} &  \Hom_R(\Omega_A(R), R) \arrow{r} &  \Hom_R(\Omega_k(R), R)  \arrow{r} \arrow[d, phantom, "\vcong"]   &  \Hom_A(\Omega_k(A), R) \arrow{r} \arrow[d, phantom, "\vcong"]  &  \Ext^1_R(\Omega_A(R), R)\\
&0\ar[equal]{u}& \Der_k(R)\arrow{r}{\beta} & \Der_k(A,R) &
\end{tikzcd}
$$
Observe that $\beta(\varepsilon_R)=\varepsilon_A$ and that $\Ext^1_R(\Omega_A(R), R)$ is annihilated by $\f$. 
Thus we obtain an embedding 
$$\begin{tikzcd}
   \Der_k(R)/\m^{-1}\varepsilon_R \arrow[hookrightarrow]{r}{\varphi}
  & \Der_k(A,R)/\m^{-1}\varepsilon_A \end{tikzcd}$$
whose cokernel is annihilated by $\f$. 

On the other hand,  the obvious inclusion of $\Der_k(A) \subset \Der_k(A,R)$ induces an $A$-linear map
\[\psi: \Der_k(A)/A\varepsilon_A \lto \ \Der_k(A,R)/\m^{-1}\varepsilon_A\, .
\]
Since the extension $A\subset R$ is birational, this map is generically injective and hence it is injective because $ \Der_k(A)/A\varepsilon_A$ is torsionfree as an $A$-module according to \autoref{1.2}.

Now we have embeddings
\begin{equation}\label{Dsed}
\begin{tikzcd}
   \Der_k(R)/\m^{-1}\varepsilon_R \arrow[hookrightarrow]{r}{\varphi}
  & \Der_k(A,R)/\m^{-1}\varepsilon_A \\
& \Der_k(A)/A\varepsilon_A \arrow[hookrightarrow]{u}{\psi} \, ,
\end{tikzcd}
\end{equation}
where the cokernels of both $\varphi$ and $\psi$ are annihilated by $\f$. Thus we obtain containments
\begin{equation}\label{incl1} \f  \cdot \im  \varphi \subset \im  \psi
\end{equation}
\begin{equation}\label{incl2}  \f \cdot \im  \psi \subset \im \varphi\, .
\end{equation}

The inclusion \autoref{incl1} shows that
\begin{align*} & \indeg \f +\findeg  ( \Der_k(R)/\m^{-1}\varepsilon_R) \ge \indeg  (\Der_k(A)/A \varepsilon_A)\\
& \findeg \f +\indeg (\Der_k(R)/\m^{-1}\varepsilon_R)  \ge \indeg ( \Der_k(A)/A \varepsilon_A)\\
& \findeg \f +\findeg  (\Der_k(R)/\m^{-1}\varepsilon_R)  \ge \findeg  ( \Der_k(A)/A \varepsilon_A) \ ;
\end{align*}
in the second inequality we use the fact that $ \Der_k(R)/\m^{-1}\varepsilon_R$ is torsionfree as an $R$-module (see \autoref{1.2}). 
The inclusion \autoref{incl2} implies the same inequalities with the roles of $ \Der_k(R)/\m^{-1}\varepsilon_R$ and $\Der_k(A)/A \varepsilon_A$ reversed. 

Finally recall that $\indeg \f=a(A)-a(R)$ and that $\findeg \f=a(A)-a(R)$ if $R$ has the generalized Cayley-Bacharach property. 
\end{proof}

\begin{proposition}\label{locirr} In addition to \autoref{sett-vectfield} assume that $R$ is reduced. Let $\p$ be a minimal prime ideal of $R$ so that $R_{\p}$ is a field and write $R'=R/\p$ and $\m'=\m/\p$. Then $\findeg (\Der_k (R)/\m^{-1}\varepsilon_R)\ge \findeg (\Der_k(R')/\m'^{-1} \varepsilon_{R'})\, .$
    \end{proposition}
\begin{proof} Since $R$ is reduced, every derivation in $\Der_k(R)$ induces a derivation in $\Der_k(R')$, see for instance \cite{ChLee}*{page 614}. This gives a natural map $\Der_k(R) \lto \Der_k(R')$.  The projection $R \twoheadrightarrow R'$ induces a map $\varphi: \Quot(R) \lto \Quot(R')$, which is surjective since $R$ is reduced. It follows that $\varphi(\m^{-1}) \subset \m'^{-1}$. Combining these facts we obtain a natural map 
$$\psi: \Der_k (R)/\m^{-1}\varepsilon_R \lto \Der_k(R')/\m'^{-1} \varepsilon_{R'}\, .$$
 Notice that $\psi_{\p}$ is an isomorphism since $R$ is reduced.  It follows that if $\nu \in  \Der_k (R)/\m^{-1}\varepsilon_R $ with $\ann_R \nu=0$, then $\ann_R'\psi(\nu)=0$, which proves the lemma.  \end{proof}

In some items of the next theorem we will assume that the curve $\C$ is locally irreducible. By this we mean that the local ring at every point of $\C$ is a domain, or equivalently, that $\C$ is the disjoint union of its irreducible components. 

\begin{theorem}[The case of curves with plane singularities]\label{curves} Let $k$ be an algebraically closed field and $\C\subset {\mathbb P_k^{n-1}}$ be a reduced curve of degree $d$. Let $R$ be the homogeneous coordinate ring of $\C$, with maximal homogeneous ideal $\m$, and let $J_R$ be the Jacobian ideal. Assume $d$ is not a multiple of the characteristic. If $\C$ has at most plane singularities, then 
\begin{eqnarray*}
\findeg (\Der_k(R)/\m^{-1}\varepsilon)  &\ge& \begin{cases} a(R)+|\Sing(\C)|-\Lmult(R/J_R) \\ 
a(R)+1 +|\Sing(\C)|-\Lmult(R/J_R)\hspace{1cm} \mbox{if } \C \mbox{ is locally irreducible } 
\end{cases}\\
&\ge& \begin{cases} a(R) +|\Sing(\C)|+\sum\limits_{p\in {\rm Sing}(\C)} {e(\mathcal O_{\C, p})-1 \choose 2}-\tau(\C) \\ 
a(R)+1 +|\Sing(\C)|+\sum\limits_{p\in {\rm Sing}(\C)} {e(\mathcal O_{\C, p})-1 \choose 2}-\tau(\C) \hspace{.2cm} \mbox{if } \C \mbox{ is locally irreducible. } 
\end{cases}\\
\end{eqnarray*}
\end{theorem}
\begin{proof} It suffices to prove the first inequality. We may assume that $n\ge 3$. Let $x,y,z$ be general linear forms in $R$ and write $A=k[x,y,z]\subset R$. Notice $A\subset R$ is  a finite and birational homogeneous extension of standard graded $k$-algebras by \autoref{UV}, and $A$ is the homogeneous coordinate ring of a plane curve $\D$. Write $J_A$ for the Jacobian ideal of $A$. 
By \autoref{GP} and \autoref{PC} one has
\[ \findeg (\Der_k(R)/\m^{-1} \varepsilon_R) \ge \indeg (\Der_k(A)/A\varepsilon_A) -a(A)+a(R)\]
and
\[\indeg (\Der_k(A)/A\varepsilon_A)\ge \begin{cases} e(A)-3 +|\Sing(\D)|-\Lmult(A/J_A) \\ 
e(A)-2 +|\Sing(\D)|-\Lmult(A/J_A)& \mbox{if } \D \mbox{ is irreducible\,.}
\end{cases}\\
 \] 
 Since $\D$ is a plane curve we have $a(A)=e(A)-3$. Thus we obtain
 \[ \findeg (\Der_k(R)/\m^{-1} \varepsilon_R) \ge \begin{cases} a(R) +|\Sing(\D)|-\Lmult(A/J_A) \\ 
a(R)+1 +|\Sing(\D)|-\Lmult(A/J_A)& \mbox{if } \D \mbox{ is irreducible\,.}
\end{cases}\\
 \]
 
 Next we show that $|\Sing(\D)|-\Lmult(A/J_A)=|\Sing(\C)|-\Lmult(R/J_R)$.  Let $\p_1, \ldots, \p_t$ be the distinct minimal prime ideals of $J_R$ having height one. As $\edim R_{\p_i}=2$, it follows that $R_{\p_i}=A_{\p_i \cap A}$, see \cite{UV}*{Proposition 5.2} for instance. In particular, $\ell\ell((R/J_R)_{\p_i})=\ell\ell((A/J_A)_{\p_i\cap A})$. The ring $A$ may acquire additional prime ideals $\q_1, \ldots,\q_s$ of height one where it is not regular, but they all correspond to ordinary nodes of $\D$, see \cite{H}*{Chapter IV, Proposition 3.5 and Theorem 3.10}, in other words $\ell\ell((A/J_A)_{\q_i})=1$. It follows that 
 \[\Lmult(A/J_A)-\Lmult(R/J_R)=s=|\Sing(\D)|-|\Sing(\C)|\, ,\] as required. This completes the proof of the first inequality if the assumption of $\C$ being local irreducible is replaced by $\C$ being irreducible. 

It remains to reduce the locally irreducible case to the irreducible case. Thus assume that $\C$ is locally irreducible and let $\wp_1, \ldots, \wp_r$ be the minimal prime ideals of $R$. Consider the exact sequence of $R$-modules
 \[ 
0\lto R \stackrel{\iota}{\lto} \times_{i=1}^{r} (R/\wp_i) \lto N \lto 0\,.
\]
Since $\C$ is locally irreducible, the map $\iota$ is an isomorphism locally on the punctured spectrum of $R,$ so $N$ is a module of finite length. It follows that $\omega_R \cong \times_{i=1}^{r} \omega_{R/\wp_i}$ and therefore $$a(R)=\max \{ a(R/\wp_i) \, |  \, 1\le i \le r\}\, . $$
Now let $\wp$ be a minimal prime of $R$ such that $a(R)=a(R/\wp)$, write $R'=R/\wp$, and let $\mathcal C'$ be the corresponding irreducible curve. We obtain
 \begin{eqnarray*}
   \findeg (\Der_k (R)/\m^{-1}\varepsilon_R)&\ge& \findeg (\Der_k(R')/\m'^{-1} \varepsilon_{R'})  \hspace{3cm} \text{by \autoref{locirr}}  \\
   &\ge& a(R')+1 +|\Sing(\C')|-\Lmult(R'/J_{R'}) \qquad \text{since $\C'$ is irreducible}  \\
   &=& a(R)+1 +|\Sing(\C')|-\Lmult(R'/J_{R'})\\
   &\ge& a(R)+1 +|\Sing(\C)|-\Lmult(R/J_{R})\, , 
 \end{eqnarray*}
 where the last inequality holds because $\C$ is the disjoint union of its irreducible components.
\end{proof}

\begin{remark} {\rm If in addition to the assumption of \autoref{curves}, the ring $R$ satisfies the generalized Cayley-Bacharach property, then according to \autoref{GP} 
 \begin{eqnarray*}\indeg (\Der_k(R)/\m^{-1}\varepsilon)     &\ge&  a(R)+|\Sing(\C)|-\Lmult(R/J_R)   \\ &\ge& a(R) +|\Sing(\C)|+\sum\limits_{p\in {\rm Sing}(\C)} {e(\mathcal O_{\C, p})-1 \choose 2}-\tau(\C) \, . \end{eqnarray*}}
\end{remark}

The next result was first proved for plane curves in \cite{CL}, then for complete intersection curves in \cite{CCF}, and finally for arithmetically Cohen-Macaulay curves in \cite{E}*{Theorem 1}.

\begin{corollary}[The case of curves with ordinary nodes]\label{curvesnodes}Let $k$ be an algebraically closed field and $\C\subset {\mathbb P_k^{n-1}}$ be a reduced curve of degree $d$. Let $R$ be the homogenous coordinate ring of $\C$ with maximal homogeneous ideal $\m$. Assume $d$ is not a multiple of the characteristic. If $\C$ has at most ordinary nodes as singularities, then 
$$\findeg (\Der_k(R)/\m^{-1}\varepsilon)  \ge \begin{cases} a(R) \\ 
a(R)+1 & \mbox{if } \C \mbox{ is locally irreducible} \, .\end{cases} $$
\end{corollary}
\begin{proof} The assertion follows from \autoref{curves}, because $\Lmult(R/J_R)=|\Sing(\C)|$ if (and only if) $\C$ has only ordinary nodes as singularities.  
\end{proof}

\begin{corollary}\label{Smooth} Let $k$ be a perfect field and $\C\subset {\mathbb P_k^{n-1}}$ be a curve of degree $d$. Let $R$ be the homogenous coordinate ring of $\C$ with maximal homogeneous ideal $\m$. Assume $d$ is not a multiple of the characteristic. If $\C$ is smooth, then 
$$\findeg (\Der_k(R)/\m^{-1}\varepsilon)  \ge   a(R)+1\, . $$
\end{corollary}

\medskip

\begin{corollary}\label{SmoothGOR} Let $k$ be a perfect field and $\C\subset {\mathbb P_k^{n-1}}$ be a curve of degree $d$. Let $R$ be the homogenous coordinate ring of $\C$ with maximal homogeneous ideal $\m$. Assume $d$ is not a multiple of the characteristic. If $\C$ is smooth and arithemetically Gorenstein, then 
$$\Der_k(R)/R\varepsilon  \cong  \m(-a(R))\, . $$
In particular, $\indeg (\Der_k(R)/R\varepsilon)=\findeg (\Der_k(R)/R\varepsilon)  =  a(R)+1\, . $
\end{corollary}
\begin{proof} First notice that $R$ is a domain, hence $\indeg (\Der_k(R)/R\varepsilon)=\findeg (\Der_k(R)/R\varepsilon)$. As in the proof of \autoref{anticanonicalM} we have an exact sequence
$$0\lto \Der_k(R)/R\varepsilon \lto \omega_R^* \lto \Ext^2_R(k,R)\, ,
$$ 
where $\Ext^2_R(k,R)$ is concentrated in degrees $\le a(R).$  Since  $\indeg (\Der_k(R)/R\varepsilon)\ge  a(R)+1$ by \autoref{Smooth}, we conclude that 
$$\Der_k(R)/R\varepsilon  \cong  (\omega_R^*)_{\ge a(R)+1}\, . $$
Now the assertion follows because, $\omega_R^*\cong R(-a(R)).$ 
\end{proof}

\bs

\section{Lower bounds in terms of algebraic and geometric genus}\label{SecEK}

The main results of this section are the estimates on the degrees of vector fields of \autoref{dPW} and \autoref{turina}. 
Our estimates will require  \autoref{regularity} below, a remarkable lower bound for the
Castelnuovo-Mumford regularity of $R/(\im \mu)^{\rm sat} \, $ that was proved in \cite{EK2002}*{4.5}. As in \cite{EK2002}, 
we deduce this bound from the nonvanishing of a map between cohomology modules. Our proof of
the nonvanishing, \autoref{Th4.3}, uses general properties of the Koszul complex and of regular differential forms,
and is different from the proofs of the corresponding results \cite{EK2002}*{2.1 and 2.2}. In \autoref{secEuler} we will apply
 \autoref{Th4.3} to obtain structural information about the module $\Der_k(R)$ and the natural map
$\, \Der_k(R)/\m^{-1}\varepsilon \lto L^* ,$ see  \autoref{splitting} and \autoref{scrolls}.

\begin{lemma}\label{preparation} Let $k$ be a field, let $d$ and $n$ be integers with $1 <d<n$, let $x_1, \ldots, x_n$ be variables over $k$, and consider the standard graded polynomial rings $A=k[x_1, \ldots, x_d]\subset D=k[x_1, \ldots, x_n]$ with homogeneous maximal ideals $\n$ and  $\N$,  respectively. By $B_{\bullet}(A)$ and $B_{\bullet}(D)$ we denote the boundaries in the Koszul complexes $K_{\bullet}(A)=K_{\bullet}(x_1, \ldots, x_d; A)$ and $K_{\bullet}(D)=K_{\bullet}(x_1, \ldots, x_n; D) .$ 
\begin{enumerate}[$($a$)$]
\item There exists a homogeneous $A$-linear map $\delta$ fitting in the commutative diagram
$$
\begin{tikzcd} H^d_{\N}(B_{d-1}(D))\arrow[rightarrow]{r}{\alpha} & H^d_{\n}(B_{d-1}(D))  \\  
   H^{d-1}_{\n}(B_{d-2}(A))  \arrow[dashed]{u}{\delta} \arrow{r}{\gamma}  & H^d_{\n}(B_{d-1}(A)) \arrow{u}{\beta}\, .
\end{tikzcd}
$$
Here $\alpha$ is the natural map arising from the fact that $\n\subset \N$, $\beta$ is induced by the morphism of complexes  $K_{\bullet}(A) \lto K_{\bullet}(D)$, and $\gamma$ is the connecting homomorphism in the long exact sequence associated to the exact sequence $0\to B_{d-1}(A) \to K_{d-1}(A)\to B_{d-2}(A) \to 0 .$ 
\item The map $\gamma$ as in item (a) is an isomorphism in degree  zero. 
\end{enumerate}
\end{lemma}
\begin{proof}
To prove part (a) we first show that the map $\alpha$ is injective. Write $B=B_{d-1}(D)$ and $\N_i=(x_1, \ldots, x_i)D$, and notice that $H^d_{\n}(B)=H^d_{\N_d}(B)$. By \cite{B}*{8.1.2}, for $n \geq i > d$ there are natural exact sequences 
$$H^{d-1}_{\N_{i-1}}(B_{x_i}) \lto H^d_{\N_i}(B) \lto H^d_{\N_{i-1}}(B) \, .$$
As $B_{x_i}\cong \oplus \, D_{x_i}$ and ${\rm grade} ( \N_{i-1} D_{x_i} )\ge i-1 \ge d$ it follows that $H^{d-1}_{\N_{i-1}}(B_{x_i}) =0$, hence $H^d_{\N_i}(B) \hookrightarrow H^d_{\N_{i-1}}(B)$ for $n \geq i > d$. This shows that $\alpha$ is injective. 


The morphism of complexes $K_{\bullet}(A) \lto K_{\bullet}(D)$ and the naturality of the long exact sequence of local cohomology gives a commutative diagram 
$$
\begin{tikzcd} H^{d-1}_{\n}(B_{d-2}(D))\arrow{r}{\varepsilon} & H^d_{\n}(B_{d-1}(D))  \\  
   H^{d-1}_{\n}(B_{d-2}(A)) \arrow{u} \arrow{r}{\gamma}  & H^d_{\n}(B_{d-1}(A)) \arrow{u}\, .
\end{tikzcd}
$$
Since $\alpha$ is an isomorphism onto its image, the existence of $\delta$ will follow once we have shown that ${\rm im } \ \varepsilon \subset {\rm im } \, \alpha .$ In fact, we are now going to prove that 
$${\rm im } \ \varepsilon ={\rm soc}_D (H^d_{\n}(B_{d-1}(D)))=  {\rm im } \, \alpha\, .$$


The acyclicity of $K_{\bullet}(A)$ and $K_{\bullet}(D)$ imply that for $0\le i\le d-2$, 
$$ \ \  \ \  \  \, H^{i+1}_{\n}(B_{i}(A))\cong H^0(k)=k\qquad  \mbox{  as graded $A$-modules, and}$$
$$H^{i+1}_{\n}(B_{i}(D))\cong H^0(k)=k \qquad   \mbox{  as graded $D$-modules}\, .$$
In particular, we have homogeneous $D$-isomorphisms
$$H^{i+1}_{\N}(B_{i}(D))\cong k \quad \mbox{for } \ 0\le i\le n-2\, . $$

The long exact sequence of local cohomology gives an exact sequence of graded $D$-modules
$$
\begin{tikzcd}[row sep={0.7em}]  H^{d-1}_{\n}(K_{d-1}(D)) \arrow{r} \arrow[d, phantom, "\veq"] &H^{d-1}_{\n}(B_{d-2}(D))\arrow{r}{\varepsilon}\arrow[d, phantom, "\vcong"]  & H^{d}_{\n}(B_{d-1}(D)) \arrow{r}&H^{d}_{\n}(K_{d-1}(D)) \arrow[d, phantom, "\vcong"]  \\
0 & k& &H^{d}_{\n}(K_{d-1}(A)) \otimes_A D \ .
\end{tikzcd}
$$
As $x_{d+1}$ is a non zerodivisor on the $D$-module $H^{d}_{\n}(K_{d-1}(A)) \otimes_A D$, this module has trivial socle. It follows that 
$${\rm soc}_D (H^d_{\n}(B_{d-1}(D)))={\rm im } \ \varepsilon  \cong k\, .$$

On the other hand,
$$k \cong H^{d}_{\N}(B_{d-1}(D))\stackrel{\alpha}\hookrightarrow H^d_{\n}(B_{d-1}(D))\, .$$
This shows that 
$$0\neq {\rm im } \, \alpha \subset {\rm soc}_D (H^d_{\n}(B_{d-1}(D)))\, .$$
This inclusion is an equality since the socle is one-dimensional.  It follows that ${\rm im } \  \varepsilon ={\rm im } \ \alpha$.

\smallskip

We prove part (b). As before, the long exact sequence of local cohomology gives an exact sequence of graded $A$-modules
$$
\begin{tikzcd}[row sep={0.7em}] H^{d-1}_{\n}(K_{d-1}(A)) \arrow{r} \arrow[d, phantom, "\veq"]  &H^{d-1}_{\n}(B_{d-2}(A))\arrow{r}{\gamma} \arrow[d, phantom, "\vcong"] & H^{d}_{\n}(B_{d-1}(A)) \arrow[d, phantom, "\vcong"]  \\
0 & k& H^{d}_{\n}(K_{d}(A)) \ .
\end{tikzcd}
$$
Since $H^{d}_{\n}(K_{d}(A))\cong H^d_{\n}(A(-d))\cong k[x_1^{-1}, \ldots, x_d^{-1}]$, we see that $\gamma$ is an isomorphism in degree zero. 
\end{proof}

Let $k$ be a field and $R$ be a standard graded Noetherian $k$-algebra with homogeneous maximal ideal $\m$. Set $\Omega:=\Omega_k(R)$. Let $K_{\bullet}=K_{\bullet}(R)$ be the Koszul complex of the functional $\Omega\lto R$ corresponding to the Euler derivation of $R$ over $k$. Write $Z_{\bullet}=Z_{\bullet}(R)$ for the cycles of $K_{\bullet}.$ If $R$ is regular, then this notation is consistent with the one introduced in \autoref{preparation}. As $Z_{\bullet}$ is a graded commutative $R$-algebra, there is a natural homomorphism 
$$\begin{tikzcd}\bigwedge^{\bullet} Z_1 \lto Z_{\bullet}\,.
\end{tikzcd}$$

Moreover, the complex $K_{\bullet}$ is acyclic if $R$ is regular. It is also acyclic if $k$ has characteristic zero, since the differential of the de Rham complex produces a $k$-linear contracting homotopy in positive internal degree. 
If $T$ is a flat $R$-algebra with $\m T=T$, then 
$K_{\bullet} \otimes_RT$  is split-exact and the map $(\wedge^{\bullet} Z_1) \otimes_RT\lto Z_{\bullet} \otimes_RT$
is an isomorphism. In particular, the kernel and cokernel of $\wedge^{\bullet} Z_1 \lto Z_{\bullet}$ have dimension zero.
 
Any morphism of positively graded Noetherian $k$-algebras $S\lto R$ induces homomorphisms of differential graded algebras $K_{\bullet}(S) \lto K_{\bullet}(R)$ and then $Z_{\bullet}(S)  \lto Z_{\bullet}(R)$.
\smallskip


\begin{theorem}\label{Th4.3} Let $k$ be a field, let $R$ and $S$ be standard graded $k$-algebras with homogeneous maximal ideals $\m$ and $\m_S$, respectively, and assume that
the multiplicity of $R$ is not  a multiple of the characteristic of $k$. If $S \to R$ is a homogeneous homomorphism that is module finite and $d:=\dim R \geq 2$, then the induced maps
$$H^d_{\m_S}(Z_{d-1}(S))_0 \lto   H^d_{\m}(Z_{d-1}(R))_0 \ \ \ \ \ \   \hbox{and} \ \ \  \ \ \ H^d_{\m_S}(\wedge^{d-1}Z_{1}(S))_0  \lto H^d_{\m}(\wedge^{d-1} Z_{1}(R))_0$$
are nonzero.
\end{theorem}


\begin{proof} Since $d\geq 2 $ and the natural maps $\wedge^{d-1} Z_1(S) \lto Z_{d-1}(S)$ and $\wedge^{d-1} Z_1(R) \lto Z_{d-1}(R)$ have zero-dimensional
kernels and cokernels, it follows that these maps become isomorphisms after applying $H^d_{\m_S}$ and $H^d_{\m}$, respectively. Thus it suffices to prove
that the first map in the statement of the theorem is not zero.

To show that this map is not zero, we may pass to the algebraic closure of $k$ to assume that $k$ is infinite and perfect. Our assumption on the multiplicity of $R$ and the associativity formula for multiplicities imply that, for some prime ideal $\p$ of $R$ with $\dim R/\p=d,$ the multiplicity $e$ of $R/\p$ is not a multiple of the characteristic of $k$. 
Write $n=\dim S$ and let $x_1, \ldots, x_n$ be general linear forms in $S$. 
We consider the polynomial subrings $A=k[x_1, \ldots, x_d]\subset D=k[x_1, \ldots, x_n]$  of $S$, and we denote their maximal ideals by $\n$ and $\N$, respectively. 
Notice that $D$ is a Noether normalization of $S$ and  $A$ is a Noether normalization of $R$ and of $R/\p$. Moreover, ${\rm rank}_A R/\p =e$ is a unit in $k$.

Since $D \subset S$ and $A \subset R$ are integral extensions, it follows that $H^i_{\m_S} \simeq H^i_{\N \, S} \simeq H^i_{\N}$ and $H^i_{\m} \simeq H^i_{\n \, R} \simeq H^i_{\n}$. Thus it remains 
to show that the map $H^d_{\N}(Z_{d-1}(S)) \lto   H^d_{\n}(Z_{d-1}(R))$ is nonzero in degree zero. Composing this map with the natural homomorphisms 
$H^d_{\N}(Z_{d-1}(D)) \lto   H^d_{\N}(Z_{d-1}(S))$ from the right and $H^d_{\n}(Z_{d-1}(R)) \lto H^d_{\n}(Z_{d-1}(R/\p))$ from the left yields a homomorphism
$H^d_{\N}(Z_{d-1}(D)) \lto   H^d_{\n}(Z_{d-1}(R/\p))$, and it suffices to prove that this last map is nonzero in degree zero. Replacing $R$ by $R/\p$ we may now assume that $R$ is a domain, 
with Noether normalization $A$.  We need to prove that 
$$H^d_{\N}(Z_{d-1}(D)) \lto   H^d_{\n}(Z_{d-1}(R))$$
is nonzero in degree zero.

Recall that the complexes $K_{\bullet}(D)$ and $K_{\bullet}(A)$ are acyclic and that $d \geq 2$. We use \autoref{preparation} and 
the natural maps $D \to R$ and $ A \to R$ to obtain a commutative diagram
$$
\begin{tikzcd} H^d_{\N}(Z_{d-1}(D))\arrow{r}{\alpha} & H^d_{\n}(Z_{d-1}(D)) \arrow{r}{g} & H^d_{\n}(Z_{d-1}(R))  \\  
   H^{d-1}_{\n}(B_{d-2}(A))  \arrow{u} \arrow{r}{\gamma}  & H^d_{\n}(Z_{d-1}(A))  \  \ ,   \arrow{u} \arrow{ur}{f}  &\ 
\end{tikzcd} 
$$
where $\gamma$ is an isomorphism in degree zero.
We need to prove that the composition $g \circ \alpha$ is not zero in degree zero. As $\gamma$ is an isomorphism in degree zero, this will 
follow once we have shown that $f$ is nonzero in degree zero.



The morphism of complexes $K_{\bullet}(A) \lto K_{\bullet}(R)$ induces a commutative diagram with exact rows
$$
\begin{tikzcd} 0 \arrow{r} & Z_{d+1}(R)\arrow{r}&  K_d(R)\arrow{r}& Z_{d-1}(R) \arrow{r} & H_{d-1}(K_{\bullet}(R))  \arrow{r} & 0\\  
&   K_d(A) \arrow{u} \arrow{r}{\cong}& Z_{d-1}(A) \arrow{u} \ \ . & & 
\end{tikzcd}
$$
The bottom map is an isomorphism because  $K_{\bullet}(A) $ is acyclic and has length $d$. Recall that the complex $K_{\bullet}(R) $ is exact locally on the punctured spectrum.   The module $ \Omega_k(R)$ has rank $d$ 
because $K \subset L$ is a separable algebraic field extension. 
Thus $Z_d(R)$ has rank zero and therefore its dimension is $<d$. In addition, the module $H_{d-1}(K_{\bullet}(R))$ has finite length, hence dimension $<d-1$. 
Thus we obtain an induced commutative diagram
$$
\begin{tikzcd} H^d_{\n}(K_d(R))\arrow{r}{\cong}& H^d_{\n}(Z_{d-1}(R)) \\  
   H^d_{\n}(K_d(A)) \arrow{u}{h} \arrow{r}{\cong}& H^d_{\n} (Z_{d-1}(A)) \arrow{u}{f}\,.
\end{tikzcd}
$$
It remains to show that $h$ is nonzero in degree zero. 

Let $K \subset L$ be the extension
of quotient fields of $A$ and $R$. This field extension has degree $e$ and is separable since $e$ is a unit in $k$. Since $A\subset R$ is a separable Noether normalization, we can consider the complementary module $$\mathfrak C_A(R)=\{ z\in L \ | \ {\rm Tr}_{L/K} (zR)\subset A \}\, ,$$ which is a finitely generated graded $R$-module.
The image of the natural map $$\wedge^d \Omega_k(R) \lto \wedge^d \Omega_k(L)=L \ dx_1 \wedge \ldots \wedge dx_d \cong L$$
is contained in $\mathfrak C_A(R)$, see for instance \cite{K}*{Theorem 9.7}. Hence we obtain a homogeneous $R$-linear map $$\gc_R: \wedge^d \Omega_k(R) \lto \mathfrak C_A(R) \ dx_1 \wedge \ldots \wedge dx_d\, .$$ Likewise we have $$\gc_A: \wedge^d \Omega_k(A)=A\ dx_1 \wedge \ldots \wedge dx_d \lto \mathfrak C_A(A) \ dx_1 \wedge \ldots \wedge dx_d=A\ dx_1 \wedge \ldots \wedge dx_d\, ,$$ 
which is the identity map. Notice that ${\rm Tr}_{L/K}(\mathfrak C_A(R))\subset A$ by definition of the complementary module. 

Now we have a diagram of homogenous $A$-linear maps
$$
\begin{tikzcd} K_d(R)=\wedge^d \Omega_k(R) \arrow{r}{\gc_R}&  \gC_A(R) \ dx_1 \wedge \ldots \wedge dx_d \arrow{d}{\frac{1}{e}\cdot {\rm Tr}_{L/K}\ dx_1 \wedge \ldots \wedge dx_d}\\  
  K_d(A)=\wedge^d \Omega_k(A) \arrow{u} \arrow{r}{\gc_A}& A\ dx_1 \wedge \ldots \wedge dx_d  \ .
\end{tikzcd}
$$
This diagram commutes, as can be seen by following the element $dx_1 \wedge \ldots \wedge dx_d \in \wedge^d \Omega_k(A)$ and using the fact that $\frac{1}{e}\cdot {\rm Tr}_{L/K}(1)=1$. 
Applying the functor $H^d_{\n}$ to this diagram, the left vertical map becomes $h$, and it suffices to prove that $H^d_{\n}(\gc_A)$ is nonzero in degree zero. However, this map is the identity map and $H^d_{\n}(A(-d))=k[x_1^{-1}, \ldots, x_d^{-1}]$, which is the field $k$ in degree zero. 
\end{proof}

In the remainder of this section we use \autoref{Th4.3} to estimate the degree of the singular locus of vector fields. These estimates in turn will lead to bounds on the degree of the vector fields themselves. 

\begin{corollary}\label{regularity}
Adopt \autoref{setupcurves}
and assume that
the degree of $\C$ is not  a multiple of the characteristic of $k$.
Let $\eta$ be a vector field on ${\mathbb P_k^{n-1}}$ of degree $m$ leaving $\C$ invariant whose singular locus
does not contain an irreducible component of $\C$. This vector field induces  a homogeneous $R$-linear map $\mu: H\to R$ of degree $m-1$ such that $\htt \im \mu>0$. Let $L=Z_1(R)$ be as in \autoref{ZANDR}.
\begin{enumerate}[$($a$)$]
\item The natural maps 
$$H^1_{\m_S}(S/\im \eta) \longrightarrow H^1_{\m}(R/\im \mu) \quad  \mbox{and}  \quad H^1_{\m}(R/\im \mu) \longrightarrow H^2_{\m}(\im \mu)$$
are both nonzero in degree $m-1 .$
\item
$\dim(R/\im \mu)=1$, $\, \reg R/(\im \mu)^{\rm sat}\ge m ,$ 
and $\, e(R/\im \mu)\ge m+1$. 
\item If $\, [H^2_{\m}(L)\large]_0 \cong k$, then $m \geq a(R) +2 .$
\end{enumerate}
\end{corollary}
\begin{proof} 
Let $Z=Z_1(S)$ be as in \autoref{ZANDR}.
According to \autoref{Th4.3} the natural map $H^2_{\m_S}(Z) \rightarrow H^2_{\m}(L)$ is not zero in degree zero.

There is a commutative diagram 
$$
\begin{tikzcd} Z \arrow{d}\arrow[twoheadrightarrow]{r}{\eta} & \im \eta \arrow{d} \\  
   H\arrow[twoheadrightarrow]{r}{\mu}  \arrow{d}& \im \mu \\
   L&
\end{tikzcd} 
$$
where the horizontal maps are homogeneous of degree $m-1$. We also have a commutative diagram with exact rows
$$
\begin{tikzcd} 0\arrow{r} &\im \eta \arrow{r} \arrow{d}& S \arrow{r} \arrow{d}&S/\im \eta \arrow{r} \arrow{d}& 0\\  
  0\arrow{r} &\im \mu \arrow{r} & R \arrow{r}&R/\im \mu \arrow{r}& 0
\end{tikzcd} 
$$
Together, these diagrams induce a commutative diagram
$$
\begin{tikzcd} &&H^2_{\m_S}(Z)\arrow{d}{g}\arrow{dl}\\
H^1_{\m_S}(S/\im \eta) \arrow{r}{\alpha} \arrow{d}&H^2_{\m_S}(\im \eta) \arrow{d}& H^2_{\m}(H) \arrow{dl}{\beta} \arrow{d}{h}\\  
H^1_{\m}(R/\im \mu) \arrow{r}{\gamma}&H^2_{\mathfrak m}(\im \mu)& H^2_{\m}(L)\\  
\end{tikzcd} 
$$

In this diagram, the two diagonal maps are homogenous of degree $m-1$,  the map $\alpha$ is bijective because ${\rm depth}\,  S \geq 3$,  the map $\beta$ is
bijective since ${\rm dim}(\ker \mu) \leq 1$ due to the assumption that $\grade \im \mu>0$, and $h$ is bijective as $L/H$ is a module of finite length (see page \pageref{ZANDR}). 
As $h\circ g$ is not zero in degree zero, the same holds for $g$. Now the diagram readily implies part (a).

From (a) we obtain, in particular, that $ \Large[ H_{\m}^1(R/\im \mu) \Large] _{m-1}\not=0\, .$ Now the assertions about dimension and regularity
in part (b) follow immediately. As to the claim about the multiplicity, $\, e(R/\im \mu)= e(R/(\im \mu)^{\rm sat})$ and the ring $R/(\im \mu)^{\rm sat}$ is Cohen-Macaulay,
hence its multiplicity is bounded below by its regularity plus 1.

In the setting of (c), the diagram shows that $\large[ H^2_{\mathfrak m}(\im \mu)\large]_{m-1} \cong k$. Thus, since $\large[ \gamma \, \large]_{m-1} \neq 0$,
this map is surjective, and then the long exact sequence of local cohomology implies that $\large[H^2_{\m}(R) \, \large]_{m-1} =0.$ Therefore $\large[H^2_{\m}(R) \, \large]_{j} =0$ for all $j \ge m-1,$ showing that  $m-1 > a(R)$ as asserted.
\end{proof}

The multiplicity estimate in \autoref{regularity}(b) can be improved substantially if the curve
$\mathcal C$ is arithmetically Gorenstein:

\begin{proposition}\label{BM} We use the hypotheses and notation of  $\autoref{regularity},$ and write $a$ for the $a$-invariant of $R$ and $p_g$ for the geometric genus of $\mathcal C.$ If $R$ is Gorenstein, then 
$$\, e(R/\im \mu)\ge \dim_k (R_{m-\delta +a+1}) + \delta -a -1 - p_g.$$
\end{proposition}

\begin{proof} We first observe that the natural map 
$S/(f_1, \ldots, f_{n-2}) \longrightarrow R$ 
is  a surjection of rings having the same dimension. Thus
$$a=a(R) \leq a(S/(f_1, \ldots, f_{n-2}))=-n+ \sum_{j=1}^{n-2}\delta_j = \delta -2.$$
This inequality and the regularity estimate in \autoref{regularity}(b) give
$$ m-\delta +a+1 \leq m \leq \reg R/(\im \mu)^{\rm sat}.$$

Again by \autoref{regularity}(b), the standard graded algebra $R/(\im \mu)^{\rm sat}$ is  one-dimensional and therefore Cohen-Macaulay. Thus its Hilbert function increases strictly up to degree $\reg R/(\im \mu)^{\rm sat}$ and is equal to the multiplicity afterward. As $e((R/(\im \mu)^{\rm sat})=e(R/\im \mu),$ it follows that
$$e(R/\im \mu) \geq\dim_k ((R/(\im \mu)^{\rm sat})_{m-\delta +a+1}) + \delta - a -1.$$

Write $t=m- \delta +1.$ It remains to prove that 
$$\dim_k((\im \mu)^{\rm sat})_{t +a} \leq p_g.$$

Indeed, \autoref{1.-2} shows that $\im \mu \cong J(-t),$ and since $R$ is Cohen-Macaulay of dimension $\ge 2$, this isomorphism induces an isomorphism
$$(\im \mu)^{\rm sat} \cong J^{\rm sat}(-t).$$

On the other hand, $J$ is contained in $J_R,$ the Jacobian ideal of 
$R.$ 
Let $\overline{R}$ denote the integral closure of $R,$ and 
$\f:=R:_R \overline{R}$ the conductor. As is classically 
known, see e.g. \cite{N}, one has $J_R \subset \f.$ 
Thus,
as $\f$ is unmixed,
$$
J^{\rm sat} \subset \f.
$$

In turn, since $R$ is Gorenstein,
$$
\f \cong 
\Hom_R(\overline{R}, R) \cong 
\Hom_R(\overline{R}, \omega_R(-a)) \cong 
\omega_{\overline{R}}(-a) .
$$

Combining these facts we conclude that
$$
\dim_k((\im \mu)^{\sat})_{t+a} = 
\dim_k(J^{\sat})_{a}\leq
\dim_k \f_a = 
\dim_k(\omega_{\overline{R}})_0 =p_g,
$$
as required
\end{proof}

\smallskip

The main application in this section, \autoref{dPW}, generalizes results of du Plessis and Wall and  of Esteves and Kleiman \cites {dPW, EK} for the case of plane curves. In this paper, it is an easy consequence of \autoref{1.-2} 
and \autoref{regularity}. Our proof was inspired by an argument in \cite{EK}*{proof of Proposition 5.2}.

\begin{lemma}\label{HF} Let $R$ be an equidimensional Noetherian standard graded algebra over a field, with $\depth\, R>0$, let $\a$ and $\b$ be homogeneous ideals of height one, and assume that $\a \cong \b(-n)$ for some $n\in \mathbb Z$. Then 
\[e(R/\a)=e(R/\b)+n \cdot e(R)\,.\]
\end{lemma}
\begin{proof} By symmetry we may assume that $n\ge 0$, and by induction one reduces to the case $n=1$. We may further suppose that the ground field is infinite, and hence there exists a linear form $x\in R$ that is non 
zerodivisor. Thus $\a\cong\b(-1)$ and $x\b$ have the same Hilbert function, and so do $R/\a$ and $R/x\b$, which gives
$e(R/\a)=e(R/x\b).$
So it suffices to show that $e(R/x\b)=e(R/\b)+e(R)$. 

Consider the exact sequence 
\[0\lto b/x\b \lto R/x\b \lto R/\b\lto 0\, .
\]
The three $R$-modules in this sequence have the same dimension, because $\b$ is in no minimal prime ideal of $R$ and therefore $\ann_R \b \subset \sqrt{0}$. Thus, 
\[e(R/x\b)=e(R/\b)+e(\b/x\b)\,.\]
On the other hand, $e(\b/x\b)=e(\b)$ since the linear form $x$ is a non zerodivisor on $\b$, and $e(\b)=e(R)$ by the associativity formula because $\b$ is in no minimal prime ideal of $R$. 
\end{proof}

The estimates in the next theorem use the multiplicity of $R/J$, where $J$ is a partial Jacobian ideal as defined in \autoref{setupcurves}. 

\begin{theorem}\label{dPW} In addition to  \autoref{setupcurves},  assume that the degree d of $\C$ is not a multiple of the characteristic. One has
$$ \findeg (\Der_k(R)/\m^{-1}\varepsilon)\ge  \begin{cases} a(R)+1 & \mbox{if } \C \mbox{ is a smooth complete intersection} \\ 
\delta-2 -\frac{e(R/J)-\delta}{d-1}& \mbox{otherwise} \, .\end{cases}
$$
\end{theorem}
\begin{proof}   If $\C$ is a smooth complete intersection
the assertion follows from \autoref{Smooth}.   Otherwise $\htt J=1\, $  by \autoref{htJ} (c). 

Let $\mu$ and $m$ be as in \autoref{1.-2} and assume that $m$ is minimal. Recall that
\[\im \mu \cong J(\delta-m-1)\]
by that corollary. We use \autoref{regularity}(b),  which says that $\htt \im \mu =1$ and $$e( R/\im \mu)\ge m+1\, .$$
Now combining \autoref{HF} with the two displayed formulas,  we obtain
\begin{equation}\label{impa}
m+1\le e(R/\im\mu)=e(R/J)+(m+1-\delta)e(R)\, ,
\end{equation}
as required.  
\end{proof}

Part (a) of the next corollary is essentially 
\cite{dPW}*{Theorem 3.2} and  \cite{EK}*{Corollary 6.4}. The estimate of part (b) is often sharper for plane curves of small genus.

\begin{corollary}\label{dPW-EK} Let $k$ be a perfect field and $\C\subset {\mathbb P_k^2}$ be a reduced curve of degree $d$ that is not smooth. Assume that d is not a multiple of the characteristic. Let $\tau$ and $p_g$ denote the total Tjurina number and the geometric genus of $\C$ and let $R$ be the homogeneous coordinate ring of $\C$.
One has
\begin{enumerate}[$($a$)$]\setlength\itemsep{2mm}
\item $\findeg (\Der_k(R)/R\varepsilon)\ge d-2 -\frac{\tau}{d-1} \, ;$
\item $\findeg(\Der_k(R)/R\varepsilon) \ge  d-\frac{3}{2} -\sqrt{2 \tau(\C)+ 2 p_g -d^2 +3d -\frac{7}{4}}\, . $
\end{enumerate}
\end{corollary}

\begin{proof} Part (a) follows from \autoref{dPW}. We prove part (b). Let $m$ be the minimal degree of a vector field leaving $\mathcal C$ invariant and recall that 
$\findeg(\Der_k(R)/R\varepsilon)=m-1.$
We start from the inequalities \autoref{impa}, but replace the multiplicity estimate of 
\autoref{regularity}(b) by the one of \autoref{BM} to obtain
$$
 \dim_k (R_{m-\delta +a+1}) + \delta -a -1 - p_g \le e(R/\im\mu)=e(R/J_R)+(m+1-\delta)\cdot e(R)\, .
$$
Notice that $\delta=d-1$, $a=d-3$, and $m-\delta +a+1=m-1 \le d-2$, where the last inequality will be proved in \autoref{UB}(d). 
Thus $\dim_k (R_{m-\delta +a+1})= \dim_k (S_{m-1})= \binom{m+1}{2},$
and we obtain $$m^2-(2d-1)m -2e(R/J_R)-2p_g+2d^2-4d+2\le 0\, .$$
Since there exists a vector field of degree $m$ whose singular locus does not contain an irreducible component of $\mathcal C$, the polynomial in $m$ on the left-hand side has a real root and the smallest real root is $$d-\frac{1}{2}-\sqrt{2 e(R/J_R)+ 2 p_g -d^2 +3d -\frac{7}{4} }\, .$$ 
\end{proof}

In addition to the assumptions of \autoref{dPW-EK} suppose that $\C$ is a rational curve, that is $p_g=0$. If moreover $\C$ has only ordinary nodes as singularities, then the lower bound in part (a) of the corollary gives $\frac{a(R)+1}{2}$, whereas the bound in (b) gives $a(R)+1$, which is the exact value for $\indeg (\Der_k(R)/R\varepsilon)$ proved in \autoref{curvesnodes}. For rational curves in general, the bound in (b) is better than the bound in (a) if and only if $\tau < {d-1 \choose 2} +\alpha-\frac{1}{2}\sqrt{3\alpha^2+\alpha}\, ,$ where $\alpha=(d-1)^2$. For a rational curve with only ordinary nodes and ordinary cusps as singularities, this inequality holds if and only if the number of cusps is less than $\alpha-\frac{1}{2}\sqrt{3\alpha^2+\alpha} . $

It will follow from \autoref{UB}(e) below that if $\C$ is a smooth complete intersection, then the equality  $\findeg (\Der_k(R)/R\varepsilon)=  a(R)+1 $ holds in \autoref{dPW}. Therefore we are not going to consider this case in the remainder of this section.

\begin{theorem}\label{linkage} Let $k$ be an algebraically closed field of characteristic zero and $X\subset {\mathbb P_k^{n-1}}$ be an equidimensional subscheme of dimension $s$, where $1\le s\le 3$. Assume that $X$ is locally a complete intersection and has only isolated singularities. If $X$ is defined scheme theoretically  by an ideal $I$ generated by forms of degrees $\le t$, let  $Z$ be a complete intersection of dimension $s$ defined by general forms of degree $t$ in $I,$   and let $Y$ be the link of $X$ with respect to $Z$. 
\begin{enumerate}[$(a)$]
\item $Y$ and $X\cap Y$ are nonsingular;
\item $\Sing(Z)$ is the disjoint union of $\Sing(X)$ and $X \cap Y$; if $p\in \Sing(X)$, then $\mathcal O_{Z,p} \cong \mathcal O_{X,p}$ , and if $p\in X\cap Y$, then $\widehat{\mathcal O_{Z,p} }\cong k \llbracket  x_1, \ldots, x_{s+1}\rrbracket /(x_1x_2)$. 
\end{enumerate}
\end{theorem}
\begin{proof} Let $S=k[x_1, \ldots, x_n]$ be the homogenous coordinate ring of ${\mathbb P_k^{n-1}},$ let $\a$ be the saturated ideal  of $Z$, and $K=\a:I$ be the saturated ideal of $Y.$ Replacing $I$ by $I_{\ge t},$ we may assume that $I$ is generated by forms $f_1, \ldots, f_m$ of degree $t.$ We write $g=\htt I=n-s-1.$ We may assume that the ideals $I$ and 
$\a$ are equal locally at every minimal prime of $I$. Therefore $\a=I\cap K,$ the ideal $K$ is unmixed of height $g,$ and all associated primes $\not=\m_s$ of $I+K$ have height $g+1,$ as can be seen from the exact sequence
$$ 0\lto S/\a \lto S/I \oplus S/K \lto S/(I+K)\lto 0\, .
$$

We now prove (a). The ideal $\a$ is generated by $g$ $k$-linear combinations $\sum \lambda_{ij} f_j$ with $\underline{\lambda}=(\lambda_{ij})$ a general point in $ \mathbb A^{gm}_k.$

We consider the polynomial rings $U=k[\{u_{ij}\,  |\,  1\le i\le g, 1\le j \le m\}]$ and  $\widetilde{S}=S\otimes_k U,$ and the $\widetilde{S}$-ideals $\widetilde{\a}=(\sum u_{ij} f_j \, |\,  1\le i\le g)$ and $\widetilde{K}=\widetilde{\a}:_{\widetilde{S} } I.$ There are natural maps 
$$\psi_1: U \lto  \widetilde{T}:=\widetilde{S}/\widetilde{K} \quad \mbox{and} \quad \psi_2: U \lto  \widetilde{P}:=\widetilde{S}/(I\widetilde{S}+\widetilde{K})\, .$$
 According to \cite{HU90}*{2.4(b)},  the generic fiber of $\psi_1$ satisfies Serre's condition $R_s$ and the generic fiber of $\psi_2$ satisfies $R_{s-1}.$ 
 
Let $Q$ be the quotient field of $U,$ and write $T_Q=\widetilde{T} \otimes_U Q$ and $P_L=\widetilde{P} \otimes_U Q.$ The rings $T_Q$ and $P_Q$ are standard graded $Q$-algebras of dimension at most $s+1$ and $s,$ respectively. Since these rings satisfy $R_s$ and $R_{s-1},$ respectively, they are regular locally on the punctured spectrum. Thus by \autoref{jak} there exits an integer $\ell$ such that
$$(x_1, \ldots, x_n)^{\ell}\, T_Q \subset \Fitt_{s+1}(\Omega_Q(T_Q)) \quad \mbox{and} \quad (x_1, \ldots, x_n)^{\ell}\, P_Q \subset \Fitt_{s}(\Omega_Q(P_Q)) \, .$$
Hence for some nonzero polynomial $h\in U,$
\begin{equation}\label{eqh}h\, (x_1, \ldots, x_n)^{\ell}\, \widetilde{T} \subset \Fitt_{s+1}(\Omega_U(\widetilde{T})) \quad \mbox{and} \quad h\, (x_1, \ldots, x_n)^{\ell}\, \widetilde{P}\subset \Fitt_{s}(\Omega_U(\widetilde{P}))\, .
\end{equation}

For  a point  $\underline{\lambda}=(\lambda_{ij})\in \mathbb A^{gm}_k,$ we write $k(\underline{\lambda})=U/(u_{ij}-\lambda_{ij}),$ $S_{\lambda}=\widetilde{S}\otimes_Uk_{\lambda},$
$T_{\underline{\lambda}}=\widetilde{T} \otimes_U k(\underline{\lambda}),$ and $P_{\underline{\lambda}}=\widetilde{P} \otimes_U k(\underline{\lambda})\, .$
It follows from  \autoref{eqh} that whenever $h(\underline{\lambda})\not=0,$ then 
$$(x_1, \ldots, x_n)^{\ell}\, T_{\underline{\lambda}} \subset \Fitt_{s+1}(\Omega_U(\widetilde{T})) \otimes_U k(\underline{\lambda}) =\Fitt_{s+1}(\Omega_k(T_{\underline{\lambda}})) $$
$$(x_1, \ldots, x_n)^{\ell}\, P_{\underline{\lambda}}\subset \Fitt_{s}(\Omega_U(\widetilde{P})) \otimes_U k(\underline{\lambda}) =\Fitt_{s+1}(\Omega_k(P_{\underline{\lambda}}))\, .
 $$
 We conclude that locally on the punctured spectrum, $\Omega_k(T_{\underline{\lambda}})$ is generated by $s+1$ elements and $\Omega_k(P_{\underline{\lambda}})$ is generated by $s$ elements. 
 
On the other hand, we may assume that $\a=\widetilde{\mathfrak a}S_{\lambda}.$ Hence there is a natural epimorphism of $S$-algebras 
$$
\begin{tikzcd}
  T_{\lambda} =S_{\lambda}/\widetilde{K}S_{\lambda} \arrow[r, twoheadrightarrow]    & S_{\lambda}/(\widetilde{\mathfrak a}S_{\lambda}: I)=S/K\, ,
   \end{tikzcd}
   $$
and likewise $P_{\lambda}\twoheadrightarrow S/(I+K).$ Thus locally on the punctured spectrum, the modules 
$\Omega_k(S/K)$ and $\Omega_k(S/(I+K))$ too are generated by $s+1$ and $s$ elements, respectively. As $S/K$ and $S/(I+K)$ are equidimensional  $k$-algebra
of dimension $s+1$ and $s$, respectively, it follows that both rings
are regular on the punctured spectrum, see \autoref{jak}.

For the proof of part (b), recall that $Z=X \cup Y$. Let $p\in Z$. If $p\notin Y$, then  $\mathcal O_{Z,p} \cong \mathcal O_{X,p}.$ If $p\notin X$, then $\mathcal O_{Z,p} \cong \mathcal O_{Y,p},$ which is regular by part (a). If $p\in X \cap Y, $ then $\mathcal O_{Z,p}$ has at least two distinct minimal primes, hence cannot be regular. This shows that $\Sing(Z)=\Sing(X) \cup (X\cap Y).$

By the general choice of $\a,$ we have $I_{\p}=\a_{\p}$ for the finitely many prime ideals $\p$ corresponding to the singular points of $X$ (here we also use the fact that $\a$ is also general in $I_{\p},$ see \cite{PTUV}*{2.5(a)}). So $\p\not\supset \a:I=K.$ Thus for every $p\in \Sing(X),$ $p\notin Y$ and so  $\mathcal O_{Z,p} \cong \mathcal O_{X,p}.$ Moreover, $\Sing(X)$ and $X\cap Y$ are disjoint. 

It remains to prove the claim about  $\mathcal O_{Z,p}$ for $p\in X\cap Y.$ By part (a) and since $p\notin \Sing(X),$ the rings  $\mathcal O_{X,p},  \mathcal O_{Y,p},  \mathcal O_{X\cap Y,p}$ are regular. Write $S'= \mathcal O_{\mathbb P^{n-1},p},$ and let $I'$, $K'$, $\a'$ be the defining ideals in $S'$ of  $\mathcal O_{X,p},  \mathcal O_{Y,p},  \mathcal O_{Z,p}.$ It suffices to find a regular system of parameters $x_1, \ldots, x_{n-1}$ of $S'$ such that  $\a'=(x_1, \ldots, x_{g-1}, x_gx_{g+1}).$ 

Recall that the ideals $I'$ and $K'$ have height $g$ and are geometrically linked by $\a'.$ Since $S'/K'$ is Gorenstein, we have $I'/\a'\cong \omega_{S'/K'}\cong S'/K'$ is cyclic, so $g-1$ generators of $\a'$ are part of a minimal generating set of $I'.$ Call these elements $x_1, \ldots, x_{g-1}.$ Since $S'/I'$ is regular, these elements are part of a regular system of parameters of $S'$ and $I'=(x_1, \ldots, x_{g})$ with $x_1, \ldots, x_{g}$ part of a regular system of parameters. Notice $\a'=(x_1, \ldots, x_{g-1},yx_g)$ for some $y\in S'.$ Now $K'=\a':I'=(x_1, \ldots, x_{g-1},y),$ so $I'+K'=(x_1, \ldots, x_{g},y).$ Since this ideal has height $g+1$ and $S'/K'+I'$ is regular, $x_1, \ldots, x_{g},y$ form a part of a regular system of parameters of $S'$, as claimed.
\end{proof}

\begin{remark}\label{differentdeg} Following the approach of  \cite{CU}*{4.4} one sees that \autoref{linkage} still  holds when $Z$ is not necessarily defined by general forms of the same degree.
\end{remark}

In the next theorem we assume that the curve $\C$ is not a smooth complete intersection because otherwise we know from \autoref{SmoothGOR} that $\findeg (\Der_k(R)/\m^{-1}\varepsilon)= a(R)+1\,.$
\begin{theorem}\label{turina} Let $k$ be an algebraically closed field of characteristic zero and $\C\subset {\mathbb P_k^{n-1}}$ be a reduced curve of degree $d$ that is locally a complete intersection. Assume that $\C$ is not a smooth complete intersection. Let $\tau$ and $p_a$ denote the total Tjurina number and the arithmetic genus of $\C$ and let $R$ be the homogeneous coordinate ring of $\C$ with maximal homogeneous ideal $\m$. 
\begin{enumerate}[$(a)$]
\item If $R$ is a domain or, more generally, $R$ has the generalized Cayley-Bacharach property, then
$$\findeg (\Der_k(R)/\m^{-1}\varepsilon)\ge \frac{d}{d-1} \, a(R) -\frac{\tau-2}{d-1}\, .$$
\item If $R$ is Cohen-Macaulay, then 
$$\findeg (\Der_k(R)/R\varepsilon)\ge \frac{2 p_a -\tau}{d-1}\, .
$$
\end{enumerate}
\end{theorem}
\begin{proof} We deduce the theorem from \autoref{dPW}. In \autoref{setupcurves} we choose elements $f_1, \ldots, f_{n-2}$ that satisfy the conclusion of \autoref{linkage} with 
$X=\C$ and $Z=V(f_1, \ldots, f_{n-2}).$
For the proof of parts (a) and (b) we are going to estimate and compute, respectively, $e(R/J)$. We write $\a=(f_1, \ldots, f_{n-2})$ and $A=S/\a$. By \autoref{linkage}, $\Sing(Z)=\Sing(\C) \cup (\C\cap Y)$ and for every $p\in \Sing(Z)$ either $\mathcal O_{Z,p}\cong \mathcal O_{\C,p} $ or else $p\in \C\cap Y$ and $\widehat{\mathcal O_{Z,p} }\cong k \llbracket  x_1, x_2 \rrbracket /(x_1x_2). $ In the latter case the Jacobian ideal $J_{\mathcal O_{Z,p}}$ of $\mathcal O_{Z,p}$ is the maximal ideal and therefore $\mathcal O_{\C,p}/J_{\mathcal O_{Z,p}}\mathcal O_{\C,p}\cong k.$ It follows that 
\begin{equation}\label{multTur}e(R/J)=\tau+\deg (\C\cap Y).\end{equation}
We are now going to estimate and compute, respectively, the degree of $\C \cap Y.$ We write $K=\a : I$ for the saturated ideal defining the link $Y.$ The subscheme $\C \cap Y$ is defined by the ideal $I+K$, and $S/(I+K)\cong R/KR, $ so $\deg (\C\cap Y)=e(R/KR).$ On the other hand, $\omega_R\cong (KR) (\delta-2).$ 

We now prove part (a). Since $R$ has the generalized Cayley-Bacharach property, we have $\findeg \omega_R=\indeg \omega_R=-a(R).$ Therefore $KR$ contains a homogeneous  $R$-regular element of degree $\delta-2-a(R).$ Since moreover $\htt KR =1$, it follows that
\begin{equation}\label{eqA}\deg (\C\cap Y)=e(R/KR)\le d(\delta-2-a(R))\, .\end{equation}
Now the assertion follows by combining \autoref{eqA}, \autoref{multTur}, and \autoref{dPW}. 

To prove part (b) we write the Hilbert series of $R$ as $H_R(t)=\frac{q(t)}{(1-t)^2}.$ Since $R$ is Cohen-Macaulay, the Hilbert series of $\omega_R$ is $H_{\omega_R}(t)= \frac{t^2q(t^{-1})}{(1-t)^2}.$ Therefore
$$H_{R/KR}=H_R-t^{\delta-2}H_{\omega_R}=\frac{Q(t)}{(1-t)^2}\, ,$$ where $Q(t)=q(t)-t^{\delta}q(t^{-1}).$ Since $\dim R/KR=1,$ we have 
\begin{equation}\label{eqB}e(R/KR)=-Q'(1)=\delta q(1)-2q'(1)=\delta e_0-2e_1\, ,\end{equation}
where $e_0=d$ and $e_1$ is the first Hilbert coefficient of $R.$ On the other hand $p_a=e_1-e_0+1.$ Now the conclusion follows from \autoref{eqB}, \autoref{multTur}, the equality $\deg (\C\cap Y)=e(R/KR),$ and \autoref{dPW}. 
\end{proof}

To illustrate the above bounds we are going to present a family of curves $\mathcal C \subset \mathbb P^{n-1}_k$ for which the inequality in Theorem \ref{linkage}(b)
is an equality, see \autoref{beauty} and in particular part (c). We will use the following lemma:

\begin{lemma}\label{genus}Let $k$ be a perfect field and $A$ be a Noetherian positively graded $k$-algebra generated by $n$ homogeneous elements of degrees 
$\delta_1, \ldots, \delta_n$ none of which is a multiple of the characteristic. Assume $A$ is a reduced complete intersection of dimension 1. 
Write $\m$ for the maximal homogeneous ideal, $a$ for the $a$-invariant, 
$J_A$ for the Jacobian ideal, and $\f$ for the conductor of $A$.

One has $J_A \cong \m (-a).$ If in addition $k$ is algebraically closed and $A$ is a domain, then $$\tau(A)= \frac{a+1}{{\rm{gcd}}(\delta_1, \ldots, \delta_n)}= 2  \, \lambda(A/\f)=2\sigma(A)\, .$$
\end{lemma}
\begin{proof} Write $A \cong S/(f_1, \dots, f_{n-1}),$ where $S=k[x_1, \ldots, x_n]$ is a positively 
graded polynomial ring with ${\rm deg} \,  x_i =\delta_i$ and
$f_i$ are homogeneous polynomials of degree $d_i$. Write $y_i$ for the image of $x_i$ in $A$, $\Theta$ for
the Jacobian matrix of $f_1, \ldots, f_{n-1}$ with entries in $A$, and $\Delta_1, \ldots, \Delta_n$ for the signed
maximal minors of $\Theta$. Since the image of the matrix $\Theta$ has rank $n-1$ over $A$ and since
both vectors 
$$\begin{bmatrix}
           y_{1} \\
           \vdots \\
           y_{n}
         \end{bmatrix} \quad \text{and} \quad   \begin{bmatrix}
           \Delta_{1} \\
           \vdots \\
           \Delta_{n}
         \end{bmatrix}
$$

\noindent
are in the kernel of $\Theta$ and their images have rank 1, it follows that these vectors are proportional,
by multiplication with a quotient of two homogeneous non-zerodivisors in $A$.
On the other hand, $\m$ is generated in degrees $\delta_i$ and $J_A$ is generated
in degrees $(\sum d_j -\sum \delta_j) + \delta_i  = a+\delta_i. $ It follows that  $J_A \cong \m(-a).$

If $k$ is algebraically closed and $A$ is a domain, then the integral closure $\overline A$ is a
graded polynomial ring $k[t],$ where $t$ has degree ${\rm{gcd}}(\delta_1, \ldots, \delta_n).$ After 
regrading we may assume that this degree is 1. We may also assume that $A \neq \overline A.$ The ring $A$ is a monomial subalgebra 
of $\overline A=k[t],$ and computing local cohomology with support in $\m$ one sees that
$t^a$ is the highest degree monomial in $\overline A \setminus A .$ Thus $\overline A \, t^{a+1}=\f,$
and it follows that 
$$a+1= \lambda(\overline A/\f) =2 \, \lambda(A/\f)=2\sigma(A),$$
 where the last two equalities hold
because $A$ is Gorenstein.

On the other hand, the isomorphism $J_A \cong \m(-a)$ reads as $J_A=\m \, t^a.$ We also recall \autoref{basicP}, which applies 
since $A$ is a complete intersection. Thus we obtain
$$ \tau(A)=\lambda(A/J_A)=\lambda(A/\m \, t^a)= \lambda(A/\m) +\lambda(\m/\m \, t^a)=a+1,$$
as required.
 \end{proof}

\begin{proposition} \label{beauty} Let $k$ be an algebraically closed field of characteristic zero and let $r$ be a positive integer.  Let  $\C\subset \mathbb P_k^{n-1}$ be the curve defined by the ideal $I$ of $S=k[x_1, \ldots, x_n]$ generated by the maximal minors of the matrix
\begin{equation}\label{scrollB}
\begin{bmatrix}
   x_{1} & x_{2} & \dots & \dots & x_{n-2} & x_{n-1} \\
 x_{2}^r &  \dots &  \dots & x_{n-2}^r & x_{n-1}^r & x_1^r+x_{n}^r
    \end{bmatrix}\, .
\end{equation}

\begin{enumerate}[$($a$)$]
\item The curve $\mathcal C$ has degree $$d=r^{n-2}+r^{n-3}+\ldots+1\, ,$$ arithmetic genus $$p_a= \frac{r}{2}((n-2)r^{n-2}-r^{n-3}-\ldots-1)=\frac{1}{2}((n-2)\, r^{n-1}-d+1)\, ,$$
geometric genus
$$p_g= \frac{r}{2}( r^{n-2}-1)\, ,$$
 total Turina number 
$$\tau=r^2((n-3)r^{n-3}-r^{n-4}-\ldots-1)=(n-3)\, r^{n-1}-d+r+1\, ,$$
and singularity degree
$$\sigma=\frac{\tau}{2}=\frac{1}{2}((n-3)\, r^{n-1}-d+r+1)\, ;$$
\item If $r>1$, the set $\{ (1: 0:  \ldots : 0 : \rho_i) \ | \ \rho_i^r=-1 \}$ is the singular locus of $\C$ and for every singular point $p$ of  $\C,$
$$\widehat{\mathcal O_{\C, p}} \cong k\llbracket  \, t^{r^{n-3}}, t^{r^{n-3}+r^{n-4}}, \ldots, t^{r^{n-3}+\, \ldots\, +1} \, \rrbracket\,; $$
\item  Let $R=S/I$ be the homogeneous coordinate ring of $\C$ and let $y_i$ denote the images of $x_i$ in $R.$ The element 
$$\sum_{i=2}^{n-1} (r^{n-2}+\ldots+r^{n-i})\, y_iy_n^{r-1} \frac{\partial}{\partial x_i} \, +d \, (y_1^r+y_n^r)\frac{\partial}{\partial x_n} \in \,  \bigoplus_{i=1}^{n} \, R \frac{\partial}{\partial x_i} 
$$
 gives a minimal generator of $\, \Der_k(R)/R\varepsilon$. In particular
 $$ \indeg (\Der_k(R)/R\varepsilon)= \findeg (\Der_k(R)/R\varepsilon)=r-1= \frac{2 p_a -\tau}{d-1}\, .$$
\end{enumerate}
\end{proposition}
\begin{proof} The formula for the degree of $\C$ follows by applying \cite{HTU} to a regular sequence of $n$ forms of degree $r$. 

We begin by proving part (b). Observe that $y_1, y_n$ is a regular sequence on $R.$ To show the claim about the singular locus, let $\p\in V(J_R)$ with $\dim R/\p=1.$ 
We claim that 
\begin{equation}\label{claimex}
(y_2, \ldots, y_{n-1}, y_1^r+y_n^r)\subset \p.
\end{equation}

To this end we first prove that $y_1\not\in \p.$ Suppose $y_1\in \p.$ Modulo $y_1R,$ the $(n-2) \times (n-2)$ 
subblock of the Jacobian matrix over $R$ that uses the partial derivatives with respect to $x_1, \ldots, x_{n-2}$ and 
the minors of  \autoref{scrollB} involving the last column
is an upper triangular matrix with $y_n^r$ along the diagonal. Hence $y_n\in \p.$ Now we see from \autoref{scrollB} that $(y_2, \ldots, y_{n-1})$ is also in $\p.$ Thus $(y_1, \ldots, y_n)\in \p$, which is impossible because $\dim R/\p=1.$

Next we show that $y_n\not\in \p.$ One easily sees that 
$$ R_{y_1}/(y_n) \cong k[x_1, x_1^{-1}, x_{n-1}]/(x_{n-1}^d-x_1^d),
$$
which is reduced because the characteristic is zero. Now suppose that $y_n \in \p.$ Since both ideals $(y_n)\subset \p$ have height one and $y_n R_{\p}$ is radical, it follows that $\p R_{\p}=(y_n)R_{\p}.$ So $R_{\p}$ is a DVR, which is impossible since $\p\in V(J_R).$

Now, the $(n-2) \times (n-2)$ subblock of the Jacobian matrix over $R$ that uses the partial derivatives with respect to $x_3, \ldots, x_n$ and 
the minors of  \autoref{scrollB} involving the first column turns out to be a lower triangular matrix with diagonal entries $ry_1y_i^{r-1}$, where 
$3 \leq i \leq n$. It follows that $y_1y_3 \cdots y_n \in \p$. Hence, as both $y_1$ and $y_n$ are not in $\p$, we obtain that $y_i\in \p$ for some $i$ with $3 \leq i \leq n-1$. Reducing modulo the ideal $(y_i)$, one sees that the maximal minors of the matrix

$$\left(
  \begin{matrix}
  y_1 & y_{2} & \dots & \dots &y_i  \\
 y_{2}^r &  \dots & \dots & y_i^r & 0 
  \end{matrix}\
  \left| \ 
  \begin{matrix}
  0 & y_{i+1} & \ldots & \dots & y_{n-1} \\
 y_{i+1}^r &  \dots & \dots & y_{n-1}^r & y_1^r+y_{n}^r
  \end{matrix}
  \right.
\right)
$$

\vskip .1in

\noindent
are in $\p$. Therefore $(y_2, \ldots, y_{n-1})\in \p$ and $ y_1(y_1^r+y_{n}^r)\in \p$. As $y_1 \not\in \p$, it follows that $$(y_2, \ldots, y_{n-1}, y_1^r+y_n^r)\subset \p\, ,$$ as asserted.

Recall that $y_1 \not\in \p$. Claim \autoref{claimex} gives the containment  $\Sing(\C)\subset \{ (1: 0: \ldots : 0 : \rho_i) \ | \ \rho_i^r=-1 \}$. To prove equality and the remaining assertion of part (b) it suffices to show that for every point  $p=(1: 0: \ldots : 0 : \rho)$ with $\rho^r=-1$ one has $\widehat{\mathcal O_{\C, p}} \cong k\llbracket  \, t^{r^{n-3}}, t^{r^{n-3}+r^{n-4}}, \ldots, t^{r^{n-3}+\, \ldots\, +1} \, \rrbracket\,.$ Writing $z_i=\frac{x_i}{x_1}$ for $2\le i \le n$ we obtain $\mathcal O_{\C, p}=k[z_2, \ldots, z_n]_{(z_2, \ldots, z_{n-1}, z_n-\rho)}/H$, where $H$ is generated by the maximal minors of the matrix
$$\left(
  \begin{matrix}
  1 & z_2 & \ldots & \dots & z_{n-1} \\
z_2^r &  \dots & \dots & z_{n-1}^r & z_n^r+1
  \end{matrix}
\right)\, .
$$
The ideal $H$ contains the element $z_n^r+1-z_2^rz_{n-1}=(z_n-\rho)v-z_2^rz_{n-1},$ where $v$ is a unit, which shows that the maximal ideal of $\mathcal O_{\C, p}$ is generated by the images of $z_2, \ldots, z_{n-1}$. Now the Cohen structure theorem gives the natural surjection 
$$
\begin{tikzcd}  \varphi: B:=k\llbracket z_2, \ldots, z_{n-1}
\, \rrbracket/ K \arrow[twoheadrightarrow]{r}  & \widehat{\mathcal O_{\C, p}}\, ,
 \end{tikzcd} 
$$
where $K$ is the ideal generated by the maximal minors of the matrix
$$\left(
  \begin{matrix}
  1 & z_2 & \ldots & \dots & z_{n-2} \\
z_2^r &  \dots & \dots & z_{n-3}^r & z_{n-1}^r
  \end{matrix}
\right)\, .
$$

On the other hand, there is a natural surjection 
$$
\begin{tikzcd} \psi: B\arrow[twoheadrightarrow]{r}  &  C:=k\llbracket  \, t^{r^{n-3}}, t^{r^{n-3}+r^{n-4}}, \ldots, t^{r^{n-3}+\, \ldots\, +1} \, \rrbracket \, . 
\end{tikzcd} 
$$
The ideal $K$ is generated by the $n-3$ elements $z_i^r-z_2^rz_{i-1}$ for $3\le i\le{n-1}$. It follows that $z_2$ is a non zerodivisor on $B$ and that $B/z_2 B$ has multiplicity $r^{n-3}$. Therefore $e(B)\le r^{n-3}=e(C)$. As $B$ and $C$ are Cohen-Macaulay rings of the same dimension, $\psi$ is an isomorphism. In particular, $B$ is a domain, which then shows that $\varphi$ is an isomorphism. This completes the proof of part (b). 

\vskip .1in

We now prove part (a). According to part (b) the total Tjurina number  and the singularity degree of $\C$ are 
 $$\tau=r \cdot \tau(B)\quad \mbox{and} \quad \sigma(\C)=r \cdot \sigma(B)\, ,$$ where $B$ is the ring defined in the proof of part (b). Write $A=k[z_2, \ldots, z_{n-1}]/K$
 where $K$ is the ideal generated by the maximal minors of the matrix
$$\left(
  \begin{matrix}
  1 & z_2 & \ldots & \dots & z_{n-2} \\
z_2^r &  \dots & \dots & z_{n-3}^r & z_{n-1}^r
  \end{matrix}
\right)\, .
$$
Giving the variables $z_i$ degree $\deg z_i:=r^{n-3}+\ldots+r^{n-i-1}\, ,$ the ideal $K$ is generated by the homogenous regular sequence $z_i^r-z_2^rz_{i-1},$ where $3\le i\le n-1.$  In particular 
$$a(A)=\sum_{i=3}^{n-1} r\cdot \deg z_i -\sum_{i=2}^{n-2}\deg z_i=(n-3)r^{n-2}-r^{n-3}-\ldots-1\, .
$$
 To compute $\tau(A)$ and $\sigma(A)$ we apply \autoref{genus}. Since $\gcd(\deg z_2, \ldots, \deg z_{n-1})=1$, it follows that $$\tau(A)=2 \cdot \sigma(A)=a(A)+1=r((n-3)r^{n-3}-r^{n-4}-\ldots-1)\, .$$ Since $B=\widehat{A}\, ,$ the asserted equality for the total Tjurina number and the singularity degree of $\C$ now follow. 

To compute the arithmetic genus of $\C,$ we pass to a rational curve  $\C'\subset \mathbb P_k^{n-1}$ with homogenous coordinate ring $R',$ so that $R$ and $R'$  have the same Hilbert function. We take $\C'$ to be the curve defined by the maximal minors of the matrix
\begin{equation}\label{rationalC}\left(
  \begin{matrix}
  x_1 & x_{2} & \dots & \dots & x_{n-1} \\
 x_{2}^r &  \dots & \dots & x_{n-1}^r & x_{n}^r
  \end{matrix}
  \right)\, .
\end{equation}
Clearly $R$ and $R'$ have the same Hilbert function.

We claim that $\C'$ is parametrized by the map $F: \mathbb P^1 \lto  \mathbb P_k^{n-1}$, where 
$$F=(s^{r^{n-2}+\ldots +r+1}\, :\,  t^{r^{n-2}}s^{r^{n-3}+\ldots +r+1} \, :\,  \ldots \ldots \, : \, t^{r^{n-2}+\ldots+r}s \, :\,  t^{r^{n-2}+\ldots +r+1})\, .
$$
Let $ C:=k[    s^{r^{n-2}+\ldots +r+1}\, , \ldots, \,  t^{r^{n-2}+\ldots+r}s ,  t^{r^{n-2}+\ldots +r+1} ] $ be the homogenous coordinate ring of the image of $F$. Since $\im F$ is a monomial curve, it is covered by two affine charts obtained by setting $t=1$ or $s=1$, respectively. If we set $t=1$, then the affine coordinate ring is the polynomial ring $k[s],$ which shows that this chart is smooth and $F$ is birational onto its image (for the latter see also \cite{KPU16}*{4.6(3)}). The other affine chart has at most one singular point, namely $(1,0,\ldots, 0)$. Since the map $F$ is birational onto its image, it follows that $\deg \im F=r^{n-2}+\ldots +r+1=d\, .$ As the multiplicity of $R'$ is also $d$, the  natural surjection 
$$
\begin{tikzcd} \phi: R'\arrow[twoheadrightarrow]{r}  & C\,  
\end{tikzcd} 
$$
shows that $\phi$ is an isomorphism. Thus $\C'$ is a rational curve with at most one singular point, namely, $p=(1,0, \ldots,  0)\, .$ 

Now 
\begin{equation}\label{argenus}p_a(\C)=p_a(\C')=p_a(\C')-p_g(\C')=\sigma(\mathcal C')=\sigma(\mathcal O_{\mathcal C', p})\, ,
\end{equation}
where the first equality obtains because $R$ and $R'$ have the same Hilbert function, the second equality holds because $\C'$ is rational, and the third equality follows from \autoref{genus2} since $\C'$ is irreducible.

To compute $\sigma(\mathcal O_{\mathcal C', p})$, we let $A$ be the coordinate ring of the affine chart obtained by setting $s=1$ and we write
$z_i=\frac{x_i}{x_1}$ for $2\le i \le n$. Notice that $A:=k[z_2, \ldots, z_n]/H$, where $H$ is generated by the maximal minors of the matrix
$$\left(
  \begin{matrix}
  1 & z_2 & \ldots & \dots & z_{n-1} \\
z_2^r &  \dots & \dots & z_{n-1}^r & z_n^r
  \end{matrix}
\right)\, .
$$
Giving the variables $z_i$ degree $\deg z_i=t^{r^{n-2}+\ldots+r^{n-i}}\, ,$ the ideal $H$ is generated by the homogenous regular sequence $z_i^r-z_2^rz_{i-1},$ where $3\le i\le n.$  In particular 
$$a(A)=\sum_{i=3}^{n} r\cdot \deg z_i -\sum_{i=2}^{n}\deg z_i=(n-2)r^{n-1}-d\, .
$$
Thus by \autoref{genus}, $\sigma(A)=\frac{a(A)+1}{2}=\frac{1}{2}((n-2)\, r^{n-1}-d+1)\,.$ On the other hand $\mathcal O_{\C', p}=A_{(z_2, \ldots, z_n)A}\, .$  Now \autoref{argenus} gives the asserted equality for $p_a(\C).$

The assertion about $p_g$ follows from \autoref{genus2} and the formulas for $p_a$ and $\sigma\, .$

\vskip .1in 
We now prove part (c). To see that the vector
$$\zeta=\sum_{i=2}^{n-1} (r^{n-2}+\ldots+r^{n-i})\, y_iy_n^{r-1} \frac{\partial}{\partial x_i} \, +d \, (y_1^r+y_n^r)\frac{\partial}{\partial x_n} \in \,  \bigoplus_{i=1}^{n} \, R \frac{\partial}{\partial x_i} 
$$
belongs to $\Der_k(R)$ one has to check that $\zeta$ is in the null space of the Jacobian matrix over the ring $R$. To show that $\zeta$ annihilates the row corresponding to the $ij$ minor of \autoref{scrollB} one uses that the same minor is zero in $R$. Also notice that $\zeta$ is not a multiple of the Euler derivation, hence its image in $\Der_k(R)/R\varepsilon$ is non zero. This element is homogenous of degree $r-1\, , $ and therefore $\indeg (\Der_k(R)/R\varepsilon)\le r-1\,.$

On the other hand, $R$ is a domain because $R$ is Cohen-Macaulay and locally a domain by part (b). Therefore $\indeg (\Der_k(R)/R\varepsilon)=\findeg (\Der_k(R)/R\varepsilon)$. 
According to \autoref{turina}(b) 
$$\findeg (\Der_k(R)/R\varepsilon)\ge \frac{2 p_a -\tau}{d-1}=r-1\, ,$$ where the last equality holds by part (a). Thus $\indeg (\Der_k(R)/R\varepsilon)=\findeg (\Der_k(R)/R\varepsilon)=r-1\, $ and the image of $\zeta$ is a minimal generator of $\Der_k(R)/R\varepsilon$.
\end{proof}

\medskip
\section{Upper bounds}\label{SecUB}

Recall that a Cohen-Macaulay  positively graded algebra $R$ over a field
 is called {\it nearly Gorenstein} if the
homogenous maximal ideal $\m$ of $R$ is contained in the trace of $\omega_R$,  the image of the  natural map $\omega_R^* \otimes \omega_R \lto R$. Clearly every  Gorenstein ring is nearly Gorenstein. 

\begin{theorem}\label{UB} Let $k$ be a perfect field and $\C\subset \mathbb P_k^{n-1}$ be a reduced curve of degree $d$ that is arithmetically Cohen-Macaulay. Let $R$ be the homogenous coordinate ring of $\C$ with maximal homogeneous ideal $\m$. One has
\begin{enumerate}[$(a)$]
\item $\indeg (\Der_k(R)/R\varepsilon) \le \max\{ \indeg \, \omega_R^*, \, a(R)+1\};$
\item $\findeg (\Der_k(R)/R\varepsilon) \le \max\{ \findeg \, \omega_R^*, \, a(R)+1\}; $
\item if $\C$ is smooth and  $d$ is not a multiple of the characteristic, then
$$\indeg (\Der_k(R)/R\varepsilon) =\findeg (\Der_k(R)/R\varepsilon) = \max\{ \indeg \, \omega_R^*, \, a(R)+1\}\, ; $$
\item if $R$ is nearly Gorenstein or, more generally,  the trace of $\omega_R$ is not contained in $\m^2$,  then 
$$\indeg (\Der_k(R)/R\varepsilon) \le a(R)+1 \, ;$$ 
\item if $R$ is  Gorenstein, then 
$$\findeg (\Der_k(R)/R\varepsilon) \le a(R)+1 \, .$$ 
\end{enumerate}

\end{theorem}
\begin{proof}  We prove parts (a), (b), and (c). Since $\depth R \ge 2$, \autoref{1.2} gives an exact sequence 
\[0\lto \Der_k(R)/R \varepsilon \lto H^*\lto \Ext^2(k,R)\, .\]
This sequence shows, in particular, that $\indeg ( \Der_k(R)/R \varepsilon)  \ge \indeg H^*$. Recall from the proof of \autoref{anticanonicalM} that $\Ext^2_R(k,R)$ is concentrated in degrees at most $a(R). $  Since $\depth H^{*} >0,$ we conclude that $$\indeg ( \Der_k(R)/R \varepsilon) \le \max\{ \indeg H^*, \, a(R)+1\}\, , $$ and likewise for the faithful initial degree.

Now parts (a) and (b) follow because 
\[H^* \hookleftarrow \omega_R^*
\]
by \autoref{3iso}(b). 

For part (c) we first notice that $\indeg (\Der_k(R)/R\varepsilon) =\findeg (\Der_k(R)/R\varepsilon)$ because $R$ is a domain and $\Der_k R/R \varepsilon$ is torsionfree by \autoref{1.2}. Furthermore $\indeg (\Der_k(R)/R\varepsilon)\ge a(R)+1$ by  \autoref{Smooth}. Thus we are done if $\max\{ \findeg \, \omega_R^*, \, a(R)+1\}=a(R)+1$. Otherwise, $\Der_k(R)/R\varepsilon \cong  \omega_R^*$ by \autoref{anticanonicalM}, and the assertion follows again.

(d) Assume that  the trace of $\omega_R$ is not contained in $\m^2$, and let $x\not=0$ be a linear form in  the trace of $\omega_R$. There exist homogenous non-zero elements $\varphi_i \in \omega_R^*$ and $w_i\in \omega_R$ such that
$x=\sum \varphi_i (w_i)$
and $\varphi_i(w_i)$ are linear forms. As $\deg w_i\ge \indeg \omega_R=-a(R)$, it follows that $\deg \varphi_i \le a(R)+1$. Thus $\indeg \omega_R^* \le a(R)+1$. Now the assertion follows from (a). 

(e) Since $R$ is Gorenstein, we have $\omega_R^*\cong R(-a(R))$. Thus $\findeg \omega_R^* =a(R)$ and the assertion follows part (b). 
\end{proof}

\begin{corollary}\label{Equality} Let $k$ be an algebraically closed field and $\C\subset \mathbb P_k^{n-1}$ be an irreducible curve of degree $d$. Let $R$ be the homogenous coordinate ring of $\C$.  Assume $d$ is not a multiple of the characteristic.  If $\C$  is arithmetically nearly Gorenstein and has at most ordinary nodes as singularities, then
$$\findeg (\Der_k(R)/R\varepsilon) = a(R)+1 \, .$$ 
\end{corollary} 
\begin{proof} This follows from \autoref{curvesnodes} and \autoref{UB}. 
\end{proof}

\begin{theorem} Let $k$ be a perfect field and $\C \subset \mathbb P_k^{n-1}$ be a reduced curve of degree $d$. Let $R$ be the homogenous coordinate ring of $\C$ with maximal homogeneous ideal $\m$. Then
\begin{enumerate}[$(a)$]
\item $\indeg (\Der_k(R)/\m^{-1} \varepsilon) \le 2d -5-a(R)\, ;$
\item
$\findeg (\Der_k(R)/R \varepsilon) \le d-2$  if  ${\rm char} \, k =0$  and  $\mathcal C$  is smooth and arithmetically 
Cohen-Macaulay.
\end{enumerate}
\end{theorem}
\begin{proof} We may assume that $k$ is infinite and $n\ge3$. Let $x_1, x_2, x_3$ be general linear forms in $R$ and consider the subalgebra $A=k[x_1,x_2,x_3]$ of $R$. Notice that $A\subset R$ is a finite and birational extension by \autoref{UV} and that $A$ is a hypersurface ring.  Thus \autoref{UB}(d) gives
\[ \findeg (\Der_k(A)/A\varepsilon_A ) \leq a(A) +1 \, .\]
Also observe that  $e(A)=e(R)$ by \autoref{UV}, hence $a(A)=e(A)-3=e(R)-3=d-3$.

Now part (a) follows because 
\[\indeg (\Der_k(R)/\m^{-1} \varepsilon_R) \leq \findeg (\Der_k(A)/A\varepsilon_A ) + a(A) -a(R)\, \]
by \autoref{GP}.
If the assumptions of part (b) are satisfied, then  $R$ is the integral closure of $A$ and $R$ is a domain. Hence every derivation of $A$ can be 
extended to a derivation of $R$, according to \cite{Seidenberg}*{Theorem, page 168}. From \autoref{Dsed} in the proof of 
\autoref{GP} we see that there are embeddings 
$$
\begin{tikzcd}
 \Der_k(A)/A\varepsilon_A   \arrow[hookrightarrow]{r}
  & \Der_k(A,R)/\m^{-1}\varepsilon_A   \arrow[r, hookleftarrow, "\varphi"]  & \Der_k(R)/\m^{-1}\varepsilon_R\, .
\end{tikzcd}
$$
Since every derivation of $A$ can be extended to a derivation of $R$, the map $\varphi$ is an isomorphism. Thus we obtain an embedding $\, \Der_k(A)/A\varepsilon_A  \hookrightarrow \Der_k(R)/\m^{-1} \varepsilon_R \,$. As $\depth R >0 $, this embedding shows that
\[\findeg (\Der_k(R)/\m^{-1} \varepsilon_R) \leq \findeg (\Der_k(A)/A\varepsilon_A )\, ,\]
which proves part (b).
\end{proof}

\vspace{.0000001cm}

\begin{proposition} Let $k$ be a perfect field and $\C\subset \mathbb P_k^{n-1}$ be a reduced curve that is arithmetically Cohen-Macaulay. Let $R$ be the homogenous coordinate ring of $\C$. One has
\[
\findeg (\Der_k(R)/R\varepsilon) \le \max\{ r(R)\cdot (a(R)+2)-2\, , \, a(R)+1\}\, .\]
 
\end{proposition}
\begin{proof} In view of \autoref{UB}(b) it suffices to prove that $\findeg \omega_R^* \le r(R) \cdot (a(R)+2)-2$. We may assume that $n\ge 3$ and that $\C\subset \mathbb P_k^{n-1}$  is non-degenerate. We write $S$ for the coordinate ring of $ \mathbb P_k^{n-1}$ and consider a minimal homogeneous free $S$-resolution  $\mathbb F_{\bullet}$ of $R$.   Since $R$ is Cohen-Macaulay, the resolution  $\mathbb F_{\bullet}$ has length $n-2$ and $F_{n-2}$ is generated in degrees at most $a(R)+n$. Moreover, $\indeg F_{n-3} \ge n-2$ because the curve is non-degenerate. It follows that the entries of $\varphi$, the last matrix in the resolution $\mathbb F_{\bullet}$, have degrees at most $a(R)+2$. 

Let $\alpha$ be a general homogeneous minimal generator of $\omega_R$ of maximal degree. Observe that ${\rm ann}_R\, \alpha=0$. Moreover, the graded module
$\omega_R/R \alpha$ is minimally generated by $r(R)-1$ homogeneous elements and is presented by the transpose of $\varphi$, with one row removed.
This is a matrix with $r(R)-1$ rows and homogeneous entries of degrees at most $a(R)+2$. The ideal $\a$ of $(r(R) -1) \times (r(R)-1)$ minors of this matrix
satisfies $\a \subset \ann_R(\omega_R/R\alpha) \subset {\sqrt \a}$, has positive grade, and is generated by forms of degrees at most $(r(R)-1)(a(R)+2)$.
Thus, there exists a homogeneous non-zerodivisor $b \in \a \subset \ann_R(\omega_R/R\alpha)$ with $\deg b \leq (r(R)-1)(a(R)+2)$. 

Now the exact sequence
$$ 0 \longrightarrow \omega_R^* \longrightarrow (R\alpha)^* \longrightarrow \Ext^1_R(\omega_R/R\alpha, R) $$
shows that 
$$b (R \alpha)^* \subset \omega_R^* \, .$$
Since $(R \alpha)^* \cong R(\deg \alpha)$ and $\deg \alpha \geq  \indeg \omega_R = -a(R)$, it follows that 
$\findeg (R\alpha)^* \leq a(R)$. As moreover $b$ is a non-zerodivisor, we conclude that
$$\findeg \omega_R^* \leq \findeg b (R\alpha)^* = \deg b + \findeg (R \alpha)^* \leq (r(R)-1)(a(R)+2)+a(R) \, ,$$
as required.
\end{proof}

We finish this section by providing the minimal graded free resolution of the module $\Der_k (R)/R\varepsilon$ for the case of a smooth arithmetically Cohen-Macaulay curve in $\mathbb P_k^3$. From this we obtain, for instance, the initial degree, the minimal number of generators, and the entire Hilbert series of  $\Der_k (R)/R\varepsilon $. In particular, we see that the upper bound of \autoref{UB}(d) fails dramatically without the  nearly Gorenstein assumption.

\begin{theorem}\label{curvein3} Let $k$ be a perfect field and $\C\subset \mathbb P_k^{3}$ be a curve of degree $d$ that is smooth and arithmetically Cohen-Macaulay. Let $R$ be the homogenous coordinate ring of $\C$ and $S=k[x_1, \ldots, x_4]$  be the homogenous coordinate ring of $\mathbb P_k^{3}$. Let
$$F_{\bullet}: 0 \lto F_2=\oplus_{j=1}^{n-1} S(-b_j) \stackrel{\varphi}{\lto} F_1=\oplus_{i=1}^{n} S(-a_i) \lto S$$
be the minimal homogenous $S$-free resolution of $R$. We may assume that $a_1\le \ldots \le a_n$. 
\begin{enumerate}[$(a)$]
\item 
\[\Der_k(R)/R\varepsilon\cong  \begin{cases}\m (-a(R))& \mbox{if } n=2  \mbox{ and d is  not a multiple of the characteristic} \\ 
 \omega_R^*& n\ge 3 \, .\end{cases}
\]
\item 
\[\indeg (\Der_k(R)/R\varepsilon)=  \begin{cases}a_1+a_2-3 & \mbox{if } n=2  \mbox{ and d is  not a multiple of the characteristic} \\ 
a_1+a_2-4& n\ge 3 \, .\end{cases}
\]
\item If $n \ge 3$, then the minimal homogenous $S$-free resolution of $\, \Der_k (R)/R \varepsilon$ is of the form
\[
\begin{tikzcd}[row sep={-1.3em}]  
&  \bigoplus\limits_{2\le j_1\le j_2 \le n-1} S(-b_{j_1}-b_{j_2}+4) &  \bigoplus\limits_{\substack{2\le j\le n-1 \\ 1\le i\le n}} S(-b_j-a_i+4)  & \\
0 \arrow[r, shorten >=8ex]   &\quad  \bigoplus  \arrow[r,  shorten >=10ex] \qquad \qquad  & \qquad \quad  \bigoplus  \arrow[r,  shorten >=3ex] \qquad \qquad & \bigoplus\limits_{1\le i_1 < i_2\le n} S(-a_{i_1}-a_{i_2}+4)\\
& \bigoplus\limits_{1\le j\le n-1} S(-b_1-b_j+4) & \bigoplus\limits_{1\le i\le n} S(-b_1-a_i+4)  &
\end{tikzcd}
\]
\item Assume $n\ge 3$ and let $\psi$ be the $n-2$ by $n$ matrix obtained by deleting the first column of $\varphi$. One has 
$\htt I_{n-2}(\psi)=3$ and
$$\Der_k(R)/R\varepsilon \cong \omega_R^*\cong \frac{I_{n-2}(\psi)}{I_{n-1}(\varphi)} (4-b_1)\, .
$$
\item Assume $n\ge 3$. Write $F_2=F_{21} \oplus F_{22}$ where $F_{21}$  is generated by the first basis element of $F_2$ and $F_{22}$ is generated by the remaining basis elements, so that $\psi: F_{22}\lto F_1, $ and let $\pi: F_2 \twoheadrightarrow F_{22}$ be the natural projection. Write $-^{\vee}=\Hom_S(-, S)\, .$ Set $b=\sum\limits_{j=1}^{n-1} b_j$ and $b'=\sum\limits_{j=2}^{n-1}b_j.$ Consider the diagram
$$
\begin{tikzcd}[row sep={4em}]  
0\arrow{r} & D_2(F_{22}) \otimes \bigwedge\limits^n F_1^{\vee}(-b') \arrow{r} &F_{22} \otimes \bigwedge\limits^{n-1} F_1^{\vee}(-b') \arrow{r} & \bigwedge\limits^{n-2} F_1^{\vee}(-b') \\
 &  \hspace{3cm} 0   \arrow[r]  & F_2 \otimes \bigwedge\limits^n F_1^{\vee} (-b)\arrow{r} \arrow[u, "-\pi \otimes \delta"] & \bigwedge\limits^{n-1} F_1^{\vee}(-b) \arrow[u, "\delta"]  \, ,
\end{tikzcd}
$$
where the first and the second row are the truncated Eagon-Northcott complexes associated to the matrices $\psi$ and $\varphi$, respectively, and $\delta$ is the differential in the Koszul complex of the sequence consisting of the entries of the first column of $\varphi$. These vertical maps give a morphism of complexes $u_{\bullet}$, and the mapping cone $C(u_{\bullet})$ is a minimal homogeneous $S$-free resolution of $\Der_k (R)/R\varepsilon.$ 
\end{enumerate}
\end{theorem}
\begin{proof} We first prove the claim about the height and the second isomorphism in part (d). Notice that 
$$F_{\bullet}^{\vee}: 0 \lto S(-4) \lto  F_1^{\vee}(-4) \stackrel{\varphi^{\vee}}{\lto}  F_2^{\vee}(-4)=\oplus_{j=1}^{n-1} S(b_j-4) $$
is a minimal homogeneous free $S$-resolution of $\omega_R$. Since $\omega_R$ is a torsionfree $R$-module and $R$ is a domain, the image $\alpha \in \omega_R$ of the first basis element of $F_2^{\vee}(-4)$  generates a submodule $R\alpha \cong  R(b_1-4).$ Notice that $\dim ( \omega_R/R\alpha ) \le 1.$ As an $S$-module, $ \omega_R/R\alpha $ is presented by the $n$ by $n-2$ matrix $\psi^{\vee}.$ It follows that $\htt I_{n-2} (\psi)=\htt I_{n-2} (\psi^{\vee}) \ge 3.$ Therefore  $\ann_S (\omega_R/R\alpha) =I_{n-2}(\psi^{\vee})=I_{n-2}(\psi)$
according to \cite{BEann}*{Theorem page 232}.

On the other hand, a shift of $\omega_R$ is isomorphic to a homogenous ideal $K$ of $R$. Let $\beta \in K$ be  the element corresponding to $\alpha.$ One has
\begin{eqnarray*}
 \omega_R^* (b_1-4)&\cong& \Hom(\omega R, R\alpha)\cong \Hom(K, R\beta)\cong R\beta:_R K  \\
 &=&\ann_R (K/R\beta) =\ann_R (\omega_R/R\alpha)=\frac{\ann_S (\omega_R/R\alpha)}{I}
 \\
 &=&\frac{I_{n-2}(\psi)}{I_{n-1}(\varphi)}\, .
\end{eqnarray*}

Next we prove parts (a) and (b), which will also completes the proof of (d). If $n=2$, the assertions follow from \autoref{SmoothGOR}. Hence we may assume that $n\ge 3$. The second isomorphism in (d) shows that $\indeg \omega_R^*=\indeg I_{n-2}(\psi)+b_1-4.$ On the other hand, $\indeg I_{n-2}(\psi)=\sum\limits_{j=2}^{n-1} b_j -\sum\limits_{i=2}^{n}a_i =a_1+a_2-b_1$. The last equality holds because $\sum\limits_{j=1}^{n-1} b_j =\sum\limits_{i=1}^{n}a_i,$ by the Hilbert-Burch theorem. We conclude that $\indeg \omega_R^*=a_1+a_2-4.$ Now parts (a) and (b) follow from \autoref{anticanonicalM} once we have shown that $a_1+a_2-4>a(R).$ 

To this end we may assume that  $b_1\le \ldots \le b_{n-1}.$ Hence $a(R)=b_{n-1}-4$ and we need to prove that $a_1+a_2>b_{n-1}$. We consider the degree matrix associated to $\varphi$, which is the $n-1$ by $n$ matrix with entries $u_{ij}=b_j-a_i.$ Notice that $\varphi_{ij}=0$ if $u_{ij}\le 0$. It easily follows that $u_{j+2,j}>0$ for all $j$ since $I$ is a prime ideal (see also \cite{GM89}*{page 3142}). As $a_2=u_{1,n-1}+ \sum\limits_{j=1}^{n-2} u_{j+2,j}$ and $n\ge 3,$ we see that $a_2>b_{n-1}-a_1$. 

We now prove (e). Since $\htt I_{n-2}(\psi)\ge 3$ by part (d), the two truncated Eagon-Northcott complexes are minimal homogeneous   $S$-free resolutions of $I_{n-2}(\psi)$ and $I_{n-1}(\varphi)$, respectively. One easily checks that $u_{\bullet}$ is a morphism of complexes. Thus $C(u_{\bullet})$ is a homogeneous $S$-free resolution of $\frac{I_{n-2}(\psi)}{I_{n-1}(\varphi)}.$ It is minimal because the matrices of the vertical maps have entries in $\m_S.$ We deduce part (c) from (e), repeatedly using the equality $\sum\limits_{j=1}^{n-1} b_j =\sum\limits_{i=1}^{n}a_i.$
\end{proof}

\medskip

\section{The Euler derivation in the module of derivation}\label{secEuler}


In proving the graded case of the Zariski-Lipman conjecture, Hochster showed that, for any Noetherian positively graded  algebra $R$ over a field of characteristic zero, the Euler derivation is  a minimal generator of $\Der_k(R)$, unless $R$ is a polynomial ring over a subalgebra \cite{Hoc}*{pg 412}. One may wonder whether the Euler derivation can generate a free direct summand. In this section we use the results from \autoref{SecEK} to address this issue and the related question of whether the natural map  $\, \Der_k(R) \to L^*$ of \autoref{1.2} can be surjective. 

\smallskip

\begin{proposition}\label{splitting} Let $R$ be a two-dimensional Noetherian standard graded algebra over a field $k$, with homogeneous maximal ideal $\m$, and assume that the multiplicity of $R$ is not a multiple of the characteristic of $k$. If $R$ is Gorenstein, then the natural map $\, \Der_k(R) \to L^*$ is not surjective and the Euler derivation does not generate a free direct summand of $\ \Der_k(R)$. 
\end{proposition}
\begin{proof} Let $S$ be a polynomial ring over $k$ of dimension $\ge 3$ with homogeneous maximal ideal $\m_S$ that maps homogeneously onto $R$. The commutative diagram with exact rows 
$$\begin{tikzcd} 0 \arrow{r} &Z \arrow{r} \arrow{d}& \Omega_k(S)  \arrow{d}\arrow{r} &\m_S \arrow{d} \arrow{r} &0 \\
0 \arrow{r} &L \arrow{r}  & \Omega_k(R) \arrow{r} &\m \arrow{r} &0              
 \end{tikzcd} $$
induces a commutative diagram 
$$\begin{tikzcd} H^1_{\m_S}(\m_S) \arrow{r}{\alpha} \arrow{d}&  H^2_{\m_S}(Z)  \arrow{d}{\beta}  \hphantom{ \ .}\\
 H^1_{\m}(\m) \arrow{r}{\gamma}  &  H^2_{\m}(L) \ .     
 \end{tikzcd} $$
Here $\alpha$ is an isomorphism since ${\rm depth}_{\m_S} \,  \Omega_k(S)\ge 3$, and $\beta$ is nonzero by \autoref{Th4.3}. We conclude that $\gamma$ is nonzero. 

We use the exact sequence of \autoref{1.2}
$$\begin{tikzcd}0 \arrow{r} &\Der_k(R)/R\, \varepsilon \arrow{r} {\delta} & L^* \arrow{r}{\nu} &\Ext^2_R(k,R)   \cong \Ext^1_R(\m,R)            
 \end{tikzcd} 
 $$
that was also induced by the 
exact sequence
$ \ 0 \to L \to  \Omega_k(R)\to \m \to 0     \      . $    
 Since $R$ is Gorenstein and $\gamma\not=0$, local duality shows that $\nu\not=0$. Thus $\delta$ is not surjective. Moreover, ${\rm depth} \Der_k(R)/R \, \varepsilon=1$ by the depth lemma because ${\rm im}\, \nu$ has depth zero. Thus $R \, \varepsilon$ cannot be a direct summand of $\Der_k(R)$, a module of depth two. 
\end{proof}

Surprisingly, if $R$ is Cohen--Macaulay, but not Gorenstein then the natural map $\Der_k(R)\to L^*$ can be surjective in dimension two. This is always the case for the coordinate rings of rational normal curves of degree $n\ge 3$ in $\mathbb P^n_k,$ as we will see in \autoref{scrolls} below.

\begin{lemma}\label{lemmascroll} Let $T$ be a standard graded Noetherian domain with ${\rm grade} \, T_{+}\ge 2$. Let $s \in \mathbb N$ be invertible in $T$ and denote the    $s$th Veronese functor by $\--^{(s)}.$ Then 
$$ {\rm Der}_{T_0}(T^{(s)})=({\rm Der}_{T_0}(T))^{(s)}\, .
$$
\end{lemma}
\begin{proof}We may assume that $T_+\not=0$ and $s\geq 2$. We write $R=T^{(s)}$ and 
consider the exact sequence
\begin{equation}\label{canonicalseq}\begin{tikzcd} T\otimes_R \Omega_{T_0}(R) \arrow[r, "\varphi"] &\Omega_{T_0}(T) \arrow{r}  &  \Omega_R(T) \arrow{r} & 0          
\end{tikzcd}
\end{equation}
We first prove that $\Supp(\Omega_R(T))\subset V(T_{+}).$ Let $d:T\to \Omega_{T_0}(T)$ denote the universal derivation. Let $\p\in \Spec(T)\setminus V(T_{+})$ 
and  $\ell$  an arbitrary linear form in $T$. We need to show that $d(\ell) \in (\im\varphi)_{\p}.$ We choose $x\in T_1\setminus \p.$  Since $s\,x^{s-1}d(x)=d(x^s)\in \im\varphi$ it follows that $d(x)\in (\im\varphi)_{\p}.$
Now the containment $(s-1)\,x^{s-2}\ell \, d(x)+x^{s-1}d(\ell)=d(\ell x^{s-1})\in \im \varphi \,$ implies that $d(\ell)\in (\im \varphi)_{\p}.$

Let $K$ and $L$ be the quotient fields of $R$ and $T$, respectively. Since $\Omega_K(L)=0$ by the above, the field extension $K \subset L$ is separable algebraic. Therefore $\varphi \otimes_T L$ is an isomorphism, which shows that ${\rm ker} \, \varphi$ is a torsion module. Again, since  $\Supp(\Omega_R(T))\subset V(T_{+})$, it follows that ${\rm grade}\, \Omega_R(T) \ge 2, $ hence ${\rm Ext}^1_T(\Omega_R(T), T)=0$. Now, dualizing the sequence \autoref{canonicalseq} into $T$ gives the identification
$$ {\rm Der}_{T_0}(T)= {\rm Der}_{T_0}(R,T)\, .
$$

Thus $\,  {\rm Der}_{T_0}(R)\subset  {\rm Der}_{T_0}(T),$ and a degree argument immediately yields the desired equality.
\end{proof}

\begin{corollary}\label{corscrolls}
Let $T=T_0[z_1, \ldots, z_t]$ be a standard graded polynomial ring with $t\ge 2$,  and let $s \ge 2$ be invertible in $T$. Then 
$ {\rm Der}_{T_0}(T^{(s)})$ is the $T^{(s)}$-submodule of $\ \Der_{T_0}(T)=\oplus_{i=1}^{t} \ T \, \frac{\partial}{\partial z_i}$ minimally generated by the  homogeneous elements 
$ z_i \frac{\partial}{\partial z_j}$ for $1\le i, j\le t$.
\end{corollary}
\begin{proof} Applying \autoref{lemmascroll} we obtain
\[ {\rm Der}_{T_0}(T^{(s)})=({\rm Der}_{T_0}(T))^{(s)}=T^{(s)} [{\rm Der}_{T_0}(T)]_0\, , 
\]
where the last equality holds because ${\rm Der}_{T_0}(T)$ is generated in degree $-1$ and $-s<-1\le 0$. 
\end{proof}

The coordinate ring of the rational normal curve of degree $n$ in $\mathbb P^n_k$ is of the form $R=S/I$, where $S=k[x_0, \ldots, x_n]$ and $I$ is the ideal generated 
by the maximal minors of the matrix
\begin{equation}\label{scrollM}
\begin{bmatrix}
    x_{0}       & x_{1} & x_{2} & \dots & \dots & x_{n-1} \\
    x_{1}       & x_{2} &  \dots &  \dots& x_{n-1} & x_{n} 
    \end{bmatrix}\, .
\end{equation}
 \vspace{0.01in}
 
\noindent
We write $y_i$ for the images of $x_i$ in $R$. 
 
\begin{proposition}\label{scrolls} Let $R$ be the coordinate ring of the rational normal curve of degree $n\ge 3$ in $\mathbb P^n_k,\ $ and assume that the characteristic of the field $k$ is not a divisor of $n$.
\begin{enumerate}[$($a$)$]
\item The module $\Der_k(R)$ is the $R$-submodule of $\ \Der_k(S,R)=\oplus_{i=0}^{n} \ R \, \frac{\partial}{\partial x_i}$ minimally generated by the following  $4$ homogeneous elements of degree $0$,
$$ \sum_{i=0}^{n-1}\, (n-i)\, y_i \, \frac{\partial}{\partial x_i}\ , \qquad  \sum_{i=0}^{n-1}\, (n-i)\, y_{i+1} \, \frac{\partial}{\partial x_i}\ ,  \qquad \sum_{i=1}^{n}\,  i \, y_{i-1} \, \frac{\partial}{\partial x_i}\ , \qquad \sum_{i=1}^{n}\,  i \, y_{i}\,  \frac{\partial}{\partial x_i}\ ;
$$
in particular, $\Der_k(R)/R\varepsilon$ is minimally generated by three homogeneous element of degree zero. 
\item $L^*\cong (y_0,y_1, y_2)(1)\, .$
\item The natural map  $\Der_k(R)\to L^*$ is surjective. 
\end{enumerate}
\end{proposition}
\begin{proof} We consider the polynomial ring $T=k[u,v]$, where the variables $u$ and $v$ are given degree $\frac{1}{n}$. By mapping $y_i$ to $u^{n-i}v^i$, one identifies $R$ with the Veronese subring $k[\{u^{n-i}v^{i}\}]=\oplus_{j \in {\mathbb Z}_{\ge0}} T_j$ of $T$. 
By \autoref{corscrolls} the $R$-submodule $\, \Der_k(R)$ of $\Der_k(T)$ is minimally generated by the elements of degree zero $u\frac{\partial}{\partial u}$, $v\frac{\partial}{\partial u}$,  $u\frac{\partial}{\partial v}$, $v\frac{\partial}{\partial v}$. 

Consider the natural map  $S \to T$ of $k$-algebras with $x_i \mapsto u^{n-i}v^i$. It induces a $T$-linear map $\Omega_k(S)\otimes T \to \Omega_k(T)$ with $dx_i \mapsto (n-i)u^{n-i-1}v^idu + i u^{n-i}v^{i-1} dv$. Dualizing into $T$, we obtain a map $\, {\rm Der}_k(T) \to {\rm Der}_k(S,T)\, $ with $\, \frac{\partial}{\partial u}\mapsto \sum_{i=0}^{n-1}\, (n-i)\, u^{n-i-1}v^i  \, \frac{\partial}{\partial x_i}\, $ and $\, \frac{\partial}{\partial v}\mapsto \sum_{i=1}^{n}\, i\, u^{n-i}v^{i-1}  \, \frac{\partial}{\partial x_i}$.

Using the identification of ${\rm Der}_k(R)$ as an $R$-submodule of ${\rm Der}_k(T) $ and ${\rm Der}_k(S,T),$

$$
\begin{tikzcd}[row sep=small]
{\rm Der}_k(T)  \arrow{rr}  &  & {\rm Der}_k(S,T)\\ 
& \arrow[ul, phantom, sloped, "\supset"]~{\rm Der}_k(R)  \arrow[ur, phantom, sloped, "\subset"]&
\end{tikzcd} 
$$
\noindent
the generators $$u\frac{\partial}{\partial u}\ ,  \qquad  v\frac{\partial}{\partial u}\ ,  \qquad   u\frac{\partial}{\partial v}\ ,  \qquad  v\frac{\partial}{\partial v}$$ of ${\rm Der}_k(R)$ become $$ \sum_{i=0}^{n-1}\, (n-i)\, y_i \, \frac{\partial}{\partial x_i}\ , \qquad  \sum_{i=0}^{n-1}\, (n-i)\, y_{i+1} \, \frac{\partial}{\partial x_i}\ ,  \qquad \sum_{i=1}^{n}\,  i \, y_{i-1} \, \frac{\partial}{\partial x_i}\ , \qquad \sum_{i=1}^{n}\,  i \, y_{i}\,  \frac{\partial}{\partial x_i}\ .
$$

We now prove (b). We may assume that $k$ is perfect. As the rational normal curve is smooth, \autoref{1.2} and 
\autoref{3iso}(b)  imply that $L^*\cong \omega_R^*.$

Since $R$ is a determinantal ring, one has $\omega_R \cong (y_0,y_1)^{n-2}(n-3).$ 
It follows that
\begin{eqnarray*}
\omega_R^*(-1) & \cong & y_0^{n-2}R:_R(y_0,y_1)^{n-2}R \\
&=& (u^{n(n-2)}T\cap R):_R  (u^n,u^{n-1}v)^{n-2} R\\
&=& (u^{n(n-2)}T:_T  (u^n,u^{n-1}v)^{n-2} R) \cap R\\
&=& (u^{n(n-2)}T:_T  (u^n,u^{n-1}v)^{n-2} T) \cap R\\
&=& (u^{n-2}T) \cap R \\
&=& (y_0, y_1, y_2).
\end{eqnarray*}

To prove part (c) recall that according to \autoref{1.2}, the natural map $\, \Der_k(R) \lto L^*$ induces an embedding $\, \Der_k(R)/R\varepsilon \hookrightarrow L^*.$ Now use that $\, \Der_k(R)/R\varepsilon \,$ is minimally generated by $3$ homogenous elements of degree zero according to (a) and $L^*$ is generated by $3$ homogeneous elements of degree zero by (b). \end{proof}

\noindent{\bf Acknolwedgment.} Part of this paper was written at
 Mathematisches Forschungsinstitut Oberwolfach (MFO), Germany,
 while the third and fifth authors participated in a Research in Pairs. Part of this paper was written also at the  Instituto de Matem\'atica Pura e Aplicada (IMPA) Rio de Janeiro, Brazil. The authors are very appreciative of the hospitality offered by MFO and IMPA.

\bibliography{references}
\end{document}